\documentclass[11pt,twoside]{article}
\usepackage{latexsym}
\usepackage{amssymb,amsbsy,amsmath,amsfonts,amssymb,amscd,bbm,comment}
\usepackage{graphicx,color}
\usepackage{dsfont, float,url}
\usepackage{cancel}
\usepackage[T1]{fontenc}
\usepackage[colorlinks,linkcolor=red,citecolor=blue]{hyperref}
\usepackage{ulem}
\newcommand{\commentout}[1]{}

\newcommand{\diff}{\mathrm{d}}

\newcommand{\R}{\mathbb{R}}
\newcommand{\1}{\mathbb{1}}

\newcommand {\eps}  {\varepsilon}
\newcommand {\ep}  {\varepsilon}
\newcommand{\dst}{\displaystyle}

\newcommand {\dv}  { {\rm div} }
\newcommand {\f}   {\frac}
\newcommand {\p}   {\partial}
\definecolor{vert}{rgb}{0,0.58,0}
\newcommand{\MD}{\color{vert}}
\newcommand{\beq}{\begin{equation}}
\newcommand{\eeq}{\end{equation}}
\newtheorem{theorem}{Theorem}

\newtheorem{remark}[theorem]{Remark} 
\newtheorem{proposition}[theorem]{Proposition}

\newcommand{\qed}{{ \hfill
                       {\unskip\kern 6pt\penalty 500 \raise -2pt\hbox{\vrule\vbox to 6pt{\hrule width 6pt
                       \vfill\hrule}\vrule} \par}   }}
\title{Multispecies cross-diffusions: from a nonlocal mean-field to a porous medium system without self-diffusion}

\author{Marie Doumic \thanks{Inria and CMAP, IP Paris, Ecole polytechnique, CNRS, 
91128 Palaiseau cedex. Email : Marie.Doumic@inria.fr}\and
Sophie Hecht\thanks{Sorbonne Universit{\'e}, CNRS, Universit\'{e} de Paris, Laboratoire Jacques-Louis Lions UMR7598, F-75005 Paris. 
Email~:~sophie.hecht@sorbonne-universite.fr} \and
Beno\^ \i t Perthame\thanks{Sorbonne Universit{\'e}, CNRS, Universit\'{e} de Paris, Inria, Laboratoire Jacques-Louis Lions UMR7598, F-75005 Paris. 
Email : Benoit.Perthame@sorbonne-universite.fr}
\and Diane Peurichard\thanks{Inria Paris, team MAMBA, Sorbonne Universit{\'e}, CNRS, Universit\'{e} de Paris, Laboratoire Jacques-Louis Lions UMR7598, F-75005 Paris. 
Email : diane.a.peurichard@inria.fr}
}
\date{\today}

\begin{document}
\maketitle
\pagestyle{plain}
\pagenumbering{arabic}

\begin{abstract} 
Systems describing the long-range interaction between individuals have attracted a lot of attention in the last years, in particular in relation with living systems. These systems are quadratic, written under the form of transport equations with a nonlocal self-generated drift.
 
We establish the localisation limit, that is the convergence of nonlocal to local systems, when the range of interaction tends to $0$. These theoretical results are sustained by numerical simulations.
 
The major new feature in our analysis is that we do not need diffusion to gain compactness, at odd with the existing literature. The central compactness result is provided by a full rank assumption on the interaction kernels. In turn, we prove existence of weak solutions for the resulting system, a cross-diffusion system of quadratic type. 

\end{abstract} 
\vskip .7cm

\noindent{\makebox[1in]\hrulefill}\newline
2010 \textit{Mathematics Subject Classification.} 35Q92, 35K55, 35R09, 35B25.
\newline\textit{Keywords and phrases.} Aggregation equation; Cross-diffusion;  Nonlocal interactions; Localisation limit; Mathematical biology; multispecies models
\section{Introduction}

Aggregation models have been used to study a wide range of living systems arising in biology and ecology, such as prey-predator system,  movement of animal herds, aggregation of cells, etc. The simplest example is the non-local one species aggregation equation
\beq 
\p_t u - \dv [ u \nabla K * u ]=0 , \qquad t \geq 0 , \, x \in \R^d,\quad  d\geq 2,
\label{eq:Intro_single}
\eeq
where $u(t,x)$ models the evolution of a single population as a function of time $t\in\R_+$ and space $x\in\R^d$, and $K(x)$ is a self-interaction potential. However, a lot of biological systems are composed of many interacting species  and a generalisation of Eq.~\eqref{eq:Intro_single} to multiple species $N\geq 1$ is written as follows:
\beq 
\p_t u^i - \dv [ u^i  \nabla \sum_{j=1}^N K^{i j} * u^j ]=0 , \qquad t \geq 0 , \, x \in \R^d, 
\label{eq:Intro_multiple}
\eeq
with $u^i(t,x)$ for $i\in [\![1,N]\!]$ the local density of the $i$-th species. Notice that we do not include diffusion, which is the major challenge of our study. The kernel functions $K^{i j}$ describe the self and cross-interaction of the species, which can be either attractive or repulsive. In the literature, typical choices for the kernel function are the attractive-repulsive Morse potentials \cite{Orsogna2006}, Gaussian potentials,  compactly-supported  or yet quadratic potentials~\cite{romanczuk2012active}. These different potentials allow to model aggregative phenomena in population dynamics such as swarming \cite{Orsogna2006,romanczuk2012active,carrillo2013new,Barre2017,Barre2018,Barre2020}.

The non-local terms in Eqs.~\eqref{eq:Intro_single} and \eqref{eq:Intro_multiple} offer various mathematical challenges. In the case of a single population, the model has been widely studied. 
Several properties are now established as existence of solutions \cite{Bertozzi2011,Carrillo2011}, well-posedness of a solution \cite{Bertozzi2011}, long time behaviour \cite{Burger2008}, possible blow-up \cite{Bertozzi2007,Bertozzi2009}. The case $N>1$, where the model includes cross-diffusion, is a more recent subject. Most of the literature considers the equation with an additional diffusion term \cite{Potts2019,bruna2017diffusion,di2018nonlinear,Giunta2021,ALASIO2022113064}. In the case without diffusion, a mathematical theory has been studied in \cite{Colombo2012,Crippa2013} for smooth kernel function, see also~\cite{carrillo2020measure} for a system with singular Newtonian potentials in dimension one. In addition, the existence and uniqueness of measure solutions for symmetrisable systems, i.e., $K^{ij}=CK^{ji}$, is proved in  \cite{Francesco2013} for $N=2$. Another existence proof is presented in~\cite{burger2022porous}, Theorem 6.1.

\medskip

Our interest here is about the {\em localisation limit} for Eqs.~\eqref{eq:Intro_single} and \eqref{eq:Intro_multiple}. Let us define, for $\eps>0$,
$$ 
K_\eps (x) = \frac{1}{\eps^d}K(\frac{x}{\eps}),  
$$
 The localisation limit consists in considering 
$\eps \rightarrow 0$, meaning that the kernel functions converge toward the Dirac delta distribution. The motivation of this limit comes from the derivation of Eqs.~\eqref{eq:Intro_single} and~\eqref{eq:Intro_multiple} from many-particle systems \cite{oelschlager1990large} when the number of particles goes to $\infty$. The localisation limit is the next natural step to recover the macroscopic system linked to the many-particle system.

For the single species model, a sketch of the proof of the localisation limit is proposed in \cite{LMG}. Additional studies consider the limit in the case where diffusion is present \cite{oelschlager1990large,PHILIPOWSKI2007526}. In particular \cite{PHILIPOWSKI2007526} presents the successive limits from an interacting particle system with Brownian motion toward a macroscopic equation without diffusion. More recently, proofs using a gradient-flow approach have also been proposed~\cite{carrillo2019blob,carrillo2023nonlocal}. A common hypothesis found in these papers is that the kernel function is an auto-convolution, i.e., there exists $\rho$ such that
\beq K(x) = \check{\rho}*{\rho}(x) , \label{hyp:Into_1}
\eeq
with $\check{\rho}(\cdot)=\rho(-\cdot)$. This hypothesis allows to recover a priori estimates from the classical entropy $\mathcal{E}(u_\eps) = \int u_\eps \ln u_\eps $. In the case of multiple species $N>1$, the localisation limit has been proven with additional diffusion \cite{ChenDausJungel_2018,JungelPortischZurek_2022}. In \cite{ChenDausJungel_2018}, the limit is proven in the case of small initial data and in \cite{JungelPortischZurek_2022}, the hypotheses made on the interaction kernel are weaker than \eqref{hyp:Into_1}, and only consider that for a given function $p$,
\beq 
\int_{\R^d} pK*p \geq 0 . \label{hyp:Into_2}
\eeq
This hypothesis, together with the diffusion term is enough to obtain estimates on the density. 
Up to our knowledge, only~\cite{burger2022porous} worked out the limit without diffusion for two species, using tools based on Wasserstein distance, a method very different from the one we propose here, with different conditions on the interaction matrix. The Cahn-Hilliard limit is studied in~\cite{carrillo2023degenerate}, following the one-species case in~\cite{Elbar_JakubJDE}. The degenerate two-species cross-diffusion system with growth is also obtained as the limit of a nonlocal system with growth and interaction via a velocity potential in~\cite{david2023degenerate}.

\medskip

The paper is organised as follows. For the sake of clarity, in Sec.~\ref{single}, we first detail the sketch of the proof proposed in \cite{LMG} in the case where the kernel function is of the form \eqref{hyp:Into_1} with $\rho$ a non-negative compactly supported function. A key point of the approach is to prove the compactness of the family $(u_\eps *\rho_\eps)_{\eps>0}$ and therefore its convergence. In Sec.~\ref{sec:multi}, we extend this localisation limit to the multiple-species case \eqref{eq:Intro_multiple}. The hypotheses made on the interaction kernel are then slightly different and consist in considering kernels  $K^{ij}$ that can be decomposed thanks to a matrix $ A:=(\alpha^{kj}) \in \mathcal{M}_{p,N}(\R)$ with $p \geq N$ and $(\rho_\eps^i)_{i \in [\![1,N]\!]} \in (L^1(\R^d))^N$ under the form
\beq  \label{hyp:Into_3}
K^{ij}(x) = \sum_{k=1}^p \alpha^{ki} \alpha^{kj} \check{\rho}^{i} * {\rho}^{j} (x), \qquad \forall (i,j) \in [\![1,N]\!]^2.
\eeq
Given this decomposition, assuming in addition that the matrix $A$ is of rank N  allows to obtain compactness for the sequence $(u^i_\eps *\rho^i_\eps)_{\eps>0}$ for all $i\in [\![1,N]\!]$ and pass to the limit in the equations. While the assumption \eqref{hyp:Into_3} is stronger than the assumption \eqref{hyp:Into_2}, it allows us to obtain the limit in the case at hand,  where there is no diffusion in the system. Additionally, considering that the functions $\rho^i$ have a bounded moment of order $\f{d+2}{2}$ permit to extend the result to non-compactly supported potentials. Finally, in Sec.~\ref{num}, we illustrate the localisation limit in 2D thanks to numerical simulations.

Note that  we assume throughout $d\geq 2,$ and let the simpler case $d=1$ to the reader.

\section{The one-species case} \label{single}

In order to explain the ideas needed for systems, we first focus on the simpler one-species case. With the assumption \eqref{hyp:Into_1}, Eq.~\eqref{eq:Intro_single} can be rewritten as
\beq
\p_t u_\eps - \dv [ u_\eps \nabla \check{\rho}_\eps*{\rho}_\eps * u_\eps ]=0 , \qquad t \geq 0 , \, x \in \R^d,
\label{eq:convol}
\eeq
where the convolution kernel $\rho_\ep$ satisfies 
\beq \label{as:rho}
\rho_\eps(x)= \frac{1}{\eps^d} \rho(\frac x \eps), \quad  \rho(x) \geq 0,  \quad \int_{ \R^d} \rho(x) \diff x =1.
\eeq
We are interested in the localisation limit of this equation, {i.e.}, the case when $\ep\to 0.$ We aim to establish that when $\ep\to 0$, the solution of \eqref{eq:convol} converges toward a solution of 
\beq
\p_t u_0 - \dv [ u_0 \nabla  u_0 ]=0 , \qquad t \geq 0 , \, x \in \R^d .
\label{eq:limit}
\eeq
Existence, uniqueness as well as the localisation limit have already been stated, and the proofs sketched, in~\cite{LMG}; we revisit this note here and provide full details for these results. To avoid technicalities and keep the proofs simple, we moreover restrict ourselves to compactly supported kernels;
this  may be relaxed by assuming $\int |x|^{\f{d+2}{2}} \rho(x)\diff x <\infty,$ see Sec.~\ref{sec:multi}.
We assume that the initial data satisfies
 \beq  \label{as:init}
{u_\eps}^0  \geq 0, \qquad \int_{ \R^d} (1+ u_\eps^0+ |x|^2+ | \ln u_\eps^0 | ) u_\eps^0(x) \diff x \leq C,
\eeq
for some constant $C$ independent of $0<\eps \leq 1$.

\begin{theorem}
For $\ep\in (0,1),$ let $\rho_\ep$ be defined by~\eqref{as:rho} and $R>0$ such that $Supp(\rho)\subset B(0,R)$.  Assuming that ${u_\eps}^0$ verifies \eqref{as:init}, letting $(u_\eps)_{\eps}$ be a weak solution of \eqref{eq:convol}, then there exists a converging subsequence (that is not relabelled) such that when $\ep \to 0,$ we have
\begin{align}
\rho_\eps * u_\eps & \rightarrow u_0 \quad \mbox{ strongly in } \quad L^p(0,T;L^
q_{\mbox{\scriptsize{loc}}}(\R^d)) \quad \forall \; 1\leq p < 2, \; 1\leq q < \f{2d}{d-2},
 \label{s1:lim_L1strong}\\ 
 u_\eps & \rightharpoonup  u_0 \quad \mbox{ weakly in } \quad L^1_{\mbox{\scriptsize{loc}}}((0,T)\times\R^d), \label{s1:lim_Cinfweaku}
\end{align}
and $u_0$ is a solution of \eqref{eq:limit} in the distributional sense.
\end{theorem}

 In Sec.~\ref{ss1.1} we present a priori estimates, in Sec.~\ref{ss1.2} we show compactness of $(\rho_\eps * u_\eps)_\eps$, and finally in Sec.~\ref{ss1.3} we show the convergence in Eq. \eqref{eq:convol}.

 \begin{remark}
     We notice that for any functions $f$ and $g$ we have the following - frequently used - identity
\begin{equation}\label{prop:convol}\int(\rho * f)(x) g(x)\diff x=\iint \rho(x-y) f(y) g(x)\diff x \diff y = \iint\check{\rho}(y-x) g(x) \diff x f(y) \diff y=\int(\check{\rho}*g)(y) f(y)\diff y.
\end{equation}
 \end{remark}

\subsection{A priori estimates} \label{ss1.1}

\begin{proposition} \label{props1:ape}
For $\ep\in (0,1),$ we assume \eqref{as:rho} and \eqref{as:init}. Let $u_\eps$ be solutions of \eqref{eq:convol}, then for all $t \in [0,T]$, $u_\eps \geq 0$ and, for some constants $C_1(T)$, $C_2(T)$ independent of $\ep$
$$ 
\int_{\R^d} |x|^2 u_\eps (t) \leq C_1 (T) \qquad \mbox{and} \qquad \int_{\R^d} u_\eps(t) \,|\ln \; u_\eps (t)|  \leq C_2 (T) .
$$
Moreover, we have the following estimates, uniform with respect to $\ep$,
\begin{align}
& u_\eps \in L^\infty(0,T;L^1(\R^d)) \label{s1:uL1}, \\
& \rho_\eps*u_\eps \in L^\infty(0,T;L^1(\R^d))  \label{s1:L1}, \\ 
& \rho_\eps*u_\eps \in L^2(0,T;H^1(\R^d))   \label{s1:H1} ,\\ 
& \rho_\eps*u_\eps \in L^2(0,T;L^{\f{2d}{d-2}}(\R^d)), \qquad {\text{for }}d\geq 3   \label{s1:Lp} ,\\ 
& \sqrt{u_\eps} \nabla K_\eps*u_\eps  \in L^2((0,T) \times\R^d ). \label{s1:entropy1}
\end{align}
For $d=2,$ the estimates are the same except~\eqref{s1:Lp} which is replaced by $\rho_\ep * u_\ep \in L^2(0,T;L^{p}(\R^d))$ for any $1\leq p<\infty,$ see also Remarks~\ref{remark4} and~\ref{remark5}.
\end{proposition}

\begin{proof}

    \noindent {\em Step 1. Positivity, $L^1$ bound and $L^2$ bound.} 
Let us first prove that $u_\ep \geq 0:$ we multiply the equation by $-\1_{u_\ep \geq 0}$ to get
\[\p_t |u_\ep |_- = \dv [|u_\ep |_- \nabla_x (\check{\rho}_\ep * \rho_\ep * u_\ep)] ,
\]
so that integrating in space we get
\[ \f{\diff}{\diff t} \dst \int_{\R^d} |u_\ep |_- \diff x = 0,
\]
hence $|u_\ep(t,\cdot)|_-=|u_\ep (0,\cdot)|=0.$ Integrating the equation in space we get that $\int u_\ep (t,x) \diff x=\int u_\ep ^0 (x) \diff x$ which gives \eqref{s1:L1} and \eqref{s1:uL1}. In addition, we can remark that
\[  \dst \f{1}{2}\f{\diff}{\diff t} \dst \int_{\R^d} |\rho_\eps * u_\ep (t,x)|^2 \diff x = \dst \int_{\R^d} (\p_t u_\ep) (\check{\rho}_\ep*\rho_\eps * u_\eps)  \diff x . \]
Thus multiplying~\eqref{eq:convol} by $\check{\rho}_\ep *\rho_\eps*u_\ep$ and integrating in space, the above equation gives
\[\begin{array}{ll}
\dst \f{1}{2}\f{\diff}{\diff t} \dst \int_{\R^d} |\rho_\eps * u_\ep (t,x)|^2 \diff x 
&= \dst \int_{\R^d} (\check{\rho}_\eps *\rho_\eps * u_\eps)\dv [ u_\eps \nabla \check{\rho}_\eps*\rho_\eps * u_\eps ] \diff x 
\\
&= - \dst \int_{\R^d}  \nabla (\check{\rho}_\eps*\rho_\eps * u_\eps) \cdot \left(u_\eps \nabla(\check{\rho}_\eps *\rho_\eps * u_\eps)\right) \diff x
\\
&=- \dst \int_{ \R^d}  u_\eps | \nabla  \check{\rho}_\eps*\rho_\eps * u_\eps|^2 \diff x, 
\end{array}\]
using~\eqref{prop:convol} and  the Green's formula. After integrating in time we have
\[
\f{1}{2} \int_{\R^d} |\rho_\eps * u_\ep (t,x)|^2 \diff x + \int_0^T \int_{ \R^d}  u_\eps | \nabla  \check{\rho}_\eps*\rho_\eps * u_\eps|^2 \diff x \leq \f{1}{2} \int_{\R^d} |\rho_\eps * u^0_\ep (x)|^2 \diff x \leq \f{1}{2} \int_{\R^d} |u^0_\ep (x)|^2 \diff x.
\]
 Hence the estimate~\eqref{s1:entropy1} is verified and $\rho_\eps * u_\ep \in L^\infty(0,T;L^2(\R^d))$.

\medskip
\noindent {\em Step 2. Second moment control.} 
We integrate the equation multiplied by the weight $|x|^2$ and find
\[\begin{array}{ll}
\dst\f{\diff }{\diff t} \dst \int_{\R^d} |x|^2 u_\eps \diff x&=-\dst\int_{\R^d} 2 [ u_\eps x\cdot \nabla \check{\rho}_\eps*\rho_\eps * u_\eps ]\diff x
\\
&\leq 2 \left(\dst\int_{\R^d}  |x|^2  u_\eps (x) \diff x \right)^{1/2} \left(\dst\int_{\R^d} u_\ep (x) |\nabla \check{\rho}_\eps*\rho_\eps * u_\eps |^2\diff x\right)^{1/2} .
\end{array}
\]
Since $W(t):= \dst\int_{\R^d}  |x|^2  u_\eps (x) \diff x $ remains positive, we may write 
\[
\f{\diff }{\diff t} W^{1/2} = \f 1 {2  W^{1/2}}\f{\diff }{\diff t} W \leq \left(\dst\int_{\R^d} u_\ep (x) |\nabla \check{\rho}_\eps*\rho_\eps * u_\eps |^2\diff x\right)^{1/2}.
\]
After integration in time we conclude
\[
 W^{1/2}(t) \leq  W^{1/2} (0)+ \int_0^t \left(\dst\int_{\R^d} u_\ep (x) |\nabla \check{\rho}_\eps*\rho_\eps * u_\eps |^2\diff x\right)^{1/2} \diff s.
\]
Applying the Cauchy-Schwarz inequality to the second term of the right-hand side and taking the square, we finally obtain
\[\dst \int_{\R^d} |x|^2 u_\eps (t,x) \diff x \leq 2 \dst \int_{\R^d} |x|^2 u_\eps (0,x) \diff x +2 t \dst\int_0^t  \int_{\R^d} u_\ep (s,x) |\nabla \check{\rho}_\eps*\rho_\eps * u_\eps |^2 (s,x) \diff x \diff s,
\]
which is finite thanks to Step~1.

\medskip
{\em Step 3. Lower bound for the entropy.} 
Let us decompose $\int_{ \R^d} u_\eps |\ln (u_\eps)|_- \diff x$ as follows:
\[\dst\int_{ \R^d} u_\eps |\ln (u_\eps)|_- \diff x=\dst \int_{ \R^d} u_\eps |\ln (u_\eps)|_- \1_{u_\ep \geq e^{-|x|^2}} \diff x + \dst\int_{ \R^d} u_\eps |\ln (u_\eps)|_- \1_{u_\ep \leq e^{-|x|^2}}\diff x=A+B
\]
We then bound each term, noticing first that if $u_\ep \geq e^{-|x|^2}$ then $ |\ln (u_\eps)|_- =-\ln (u_\eps) \1_{u_\ep \leq 1} \leq |x|^2:$
\[
A=\dst\int_{ \R^d} u_\eps |\ln (u_\eps)|_- \1_{u_\ep \geq e^{-|x|^2}} \diff x
\leq 
\dst\int_{ \R^d} |x|^2 u_\eps \diff x <\infty,
 \]
 and, noticing now that if $x\geq 1$ and if $y \leq e^{-|x|^2}$ then $ y |\ln (y)| \leq e^{-|x|^2} |\ln (e^{-|x|^2})|$ because $s\mapsto s|\ln(s)|$ is increasing on $(0,e^{-1})$ (with a maximum on $e^{-1}$)
 \[B=\dst \int_{ \R^d} u_\eps |\ln (u_\eps)|_- \1_{u_\ep \leq e^{-|x|^2}}\diff x
  \leq \dst \int_{|x|\leq 1} u_\eps |\ln (u_\eps)|_- \1_{u_\ep \leq 1}\diff x
  +
  \dst \int_{ |x|\geq 1} |x|^2 e^{-|x|^2} \diff x <\infty.
 \]

 \medskip
{\em Step 4. Space compactness .} 
We consider the classical entropy $ \int_{ \R^d} u_\eps \ln (u_\eps)$ and compute its derivative
\[\begin{array}{lll}
\dst \f{\diff }{\diff t} \dst \int_{\R^d} u_\ep \ln ( u_\ep)  \diff x&=&  \dst \int_{\R^d} \left(1+\ln(u_\ep)\right)\dv [ u_\eps \nabla \check{\rho}_\eps*\rho_\eps * u_\eps ] \diff x 
\\
&=& - \dst \int_{\R^d} \nabla \left(1+\ln(u_\ep)\right)\cdot [ u_\eps \nabla \check{\rho}_\eps*\rho_\eps * u_\eps ] \diff x 
\\
&=& -\dst \int_{\R^d} \nabla u_\ep \cdot  \nabla (\check{\rho}_\eps*\rho_\eps * u_\eps ) \diff x 
\\
&=& - \dst \int_{\R^d} (\nabla u_\ep *\rho_\ep )\cdot  \nabla (\rho_\eps * u_\eps ) \diff x 
\\
&=&-\dst \int_{\R^d} |\nabla u_\ep *\rho_\ep |^2 \diff x. 
\end{array}
\]
Then integrating in time, we find
\[  \int_{ \R^d} u_\eps |\ln (u_\eps)| +    \int_0^T \int_{ \R^d} | \nabla \rho_\eps*u_\eps |^2 \diff x \diff t \leq \int_{ \R^d} u_\eps^0 \ln (u_\eps^0) +2  \int_{ \R^d} u_\eps |\ln (u_\eps)|_-,
\]
which gives the entropy bound and estimate \eqref{s1:H1}. Additionally, thanks to  the Gargliano-Nirenberg-Sobolev interpolation inequality, see \cite{LiebLoss}, we have for $\theta \in [0,1]$ and $p$ such that $\f{1}{p} = \theta (\f{1}{2}-\f{1}{d}) + \f{1-\theta}{q}$ with $q=2:$
\[
\Vert \rho_\ep * u_\ep \Vert_{L^p(\R^d)} \leq C \Vert \nabla \rho_\ep * u_\ep \Vert_{L^2(\R^d)}^\theta \Vert \rho_\ep * u_\ep \Vert_{L^q(\R^d)}^{1-\theta}.
\]
Then if $d\neq2$, for $\theta=1$ we have
\[
\| \rho_\eps*u_\eps \|_{L^2(0,T;L^{\f{2d}{d-2}}(\R^d))} \leq \| \nabla \rho_\eps*u_\eps \|_{L^2(0,T;L^2(\R^d))}
\]
which shows the estimate~\eqref{s1:Lp}. This concludes the proof of Proposition~\ref{props1:ape}.

\end{proof}

\begin{remark}\label{remark4}
Taking $\theta<1,$ we get $\rho_\ep * u_\ep \in L^{\f{2}{\theta}}_t(L^p_x)$ with $\f{1}{p}=\theta(\f{1}{2}-\f{1}{d})+1-\theta,$ hence $\rho_\ep *u_\ep \in L^p((0,T)\times \R^d)$ for $p=2+\f{2}{d}.$
\end{remark}

\begin{remark}\label{remark5}
For $d=2$ we take $\theta <1$ and $q=1$ hence $p=\f{1}{1-\theta} >1.$ We get $\rho_\ep * u_\ep \in L^{\f{2}{\theta}}(0,T; L^p(\R^d))$. For $\theta = 2/3$ we have $\rho_\ep * u_\ep \in L^3_t(L^3_x).$
\end{remark}

\subsection{Compactness} \label{ss1.2}

Our next step is to prove space and time compactness. We first remark that, thanks to the bounds \eqref{s1:L1} and \eqref{s1:H1}, we have
\beq \label{s1:L1loc}
\rho_\eps * u_\eps \in L^\infty(0,T;L^1_{\mbox{\scriptsize{loc}}}(\R^d)), \quad \nabla \rho_\eps * u_\eps \in L^1_{\mbox{\scriptsize{loc}} }((0,\infty)\times\R^d), 
\eeq
which provides us with space compactness. To pass to the limit $\eps \rightarrow 0$ we need to obtain some compactness in time. To this aim, we compute the equation of $\rho_\eps * u_\eps$,
\begin{equation}\label{eq:rho*u}
\p_t \rho_\eps* u_\eps = \dv [\underbrace{ \rho_\eps* u_\eps \nabla (\check{\rho}_\eps*{\rho}_\eps * u_\eps)}_{:=Q_\eps(t,x)}],
\end{equation}
with  $Q_\eps(t,x) = \int \rho_\eps (x-y) u_\eps (y) \nabla \check{\rho}_\eps *  {\rho}_\eps * u_\eps (t,y) \diff y $ and
\[\begin{aligned}
\|  Q_\eps(t,x)  \|_{L^1(\R^d)}  &\leq  {  \|  \rho_\eps  \|_{L^1(\R^d)}}   \int \  u_\eps (y) | \nabla \check{\rho}_\eps *  {\rho}_\eps * u_\eps (t,y) | \diff y
\\ &
 \leq  \frac{1}{2} \Big( \|   \sqrt{ u_\eps } \|^2_{ L^2(\R^d)} + \|  \sqrt{ u_\eps } | \nabla \check{\rho}_\eps *  {\rho}_\eps * u_\eps (t,y) |  \|^2_{ L^2(\R^d)} \Big). 
\end{aligned}
 \]
Thus
 \[  \|  Q_\eps(t,x)  \|_{L^1(0,T;L^1(\R^d))} \leq  \frac{1}{2} \Big( T \|  u_\eps \|^2_{L^\infty(0,T;L^2(\R^d))} + \|  \sqrt{ u_\eps } | \nabla \check{\rho}_\eps *  {\rho}_\eps * u_\eps (t,y) |  \|^2_{ L^2(0,T;L^1(\R^d))} \Big)  \leq CT. \]

Let us show that $\rho_\ep * u_\ep$ is compact in $L^1_{{\text{loc}}}((0,T)\times \R^d).$ We first have, for fixed $t$, and on a compact $K\subset \R^d,$ thanks to the inequality~\eqref{s1:H1} and to the Taylor-Lagrange inequality
 \[\Vert \rho_\ep * u_\ep (t,x+h) - \rho_\ep * u_\ep (t,x) \Vert_{L^1((0,T)\times K)} \leq C(K,T)h \Vert \nabla u_\ep * \rho_\ep \Vert_{L^2((0,T)\times K)} \leq C(K,T)h.
 \] 
 Then we take a mollifier sequence $\omega_\eta(x)=\f{1}{\eta^d}\omega(\f{x}{\eta})$ and write, for $v=u_\ep * \rho_\ep,$
 \begin{align*}\Vert v(t+k,x)-v(t,x)\Vert_{L^1((0,T)\times K)} &\leq  & \Vert (v-v*\omega_\eta)(t+k,x) - (v-v*\omega_\eta)(t,x)\Vert_{L^1((0,T)\times K)} 
 \\
 &&
 +  \Vert v*\omega_\eta (t+k,x)-v*\omega_\eta (t,x)\Vert_{L^1((0,T)\times K)}. 
 \end{align*}
 We evaluate the second term on the right-hand side using~\eqref{eq:rho*u}
 \begin{equation*}
 \begin{array}{lll}
 \Vert v*\omega_\eta (t+k,x)-v*\omega_\eta (t,x)\Vert_{ L^1((0,T)\times K)} &=&  \Vert  \int_t^{t+k}  \p_t v*\omega_\eta (t+s,x)\diff s\Vert_{L^1((0,T)\times K)} 
 \\
&=&  \Vert  \int_t^{t+k}  \dv (Q_\ep) *\omega_\eta (t+s,x)\diff s\Vert_{L^1((0,T)\times K)} 
\\
&=&\Vert  \int_t^{t+k}  Q_\ep *\dv(\omega_\eta) (t+s,x)\diff s\Vert_{L^1((0,T)\times K)} 
\\
&\leq& \int_t^{t+k}  \Vert Q_\ep *\dv(\omega_\eta) (t+s,x)\Vert_{L^1((0,T)\times K)} \diff s
\\
&\leq& \int_t^{t+k}  \f{1}{\eta} \Vert Q_\ep (t+s,\cdot) \Vert_{L^1((0,T)\times K)} \Vert \omega (t+s,\cdot)\Vert_{W^{1,1}((0,T)\times K)} \diff s
\\
&\leq& C(K) \f{k}{\eta} 
 \end{array}
 \end{equation*}
 For the first term of the right-hand side, we apply the Cauchy-Schwarz inequality, use~\eqref{s1:H1} and a classical convolution inequality:
  \[
 \Vert (v-v*\omega_\eta)(t+k,x) - (v-v*\omega_\eta)(t,x)\Vert_{L^1((0,T)\times K)} \diff t
\leq 2 \eta C(\alpha) \sqrt{|K|(T+k)}\Vert \nabla v \Vert_{L^2((0,T+k)\times K)} \leq C(K,T)\eta 
  \]
  so finally choosing for instance $\eta=\sqrt{k}$ we have satisfied the assumptions of the Weil-Kolmogorov-Frechet theorem on $L^1((0,T)\times K),$ hence the sequence $\rho_\ep * u_\ep$ is compact in this space.

\subsection{Convergence} \label{ss1.3}

Thanks to the compactness obtained in the previous section we have for a subsequence \[\begin{aligned}
\rho_\eps * u_\eps & \rightarrow u_0 \quad \mbox{ strongly in } \quad L^1_{\text{loc}}((0,T)\times\R^d), 
\\ 
\nabla \rho_\eps * u_\eps & \rightharpoonup  \nabla u_0 \quad \mbox{ weakly in } \quad L^1_{\text{loc}}((0,T)\times\R^d),
\\ 
\p_t  \rho_\eps * u_\eps & \rightharpoonup  \p_t  u_0 \quad \mbox{ weakly in } \quad L^1_{\text{loc}}((0,T)\times\R^d).
\end{aligned}
 \]

  In addition, from the uniform bound  $\rho_\eps*u_\eps \in L^2(0,T;L^{\f{2d}{d-2}}(\R^d))$ thanks to estimates \eqref{s1:Lp}, and the strong limit \eqref{s1:lim_L1strong}, we have
 \beq \label{s1:lim_Lpstrong}
 \rho_\eps * u_\eps \rightarrow u_0 \quad \mbox{ strongly in } \quad L^p(0,T;L^
q_{\mbox{\scriptsize{loc}}}(\R^d)) \quad \forall 1\leq p < 2, \; 1\leq q < \f{2d}{d-2}.
 \eeq
 
It remains to pass to the limit in the equation. Let  $\phi \in C_c^\infty([0,T]\times\R^d)$, then multiplying Eq.~\eqref{eq:convol} by $\phi$ and integrating by parts, we have 
\[
 \int_0^T \int_{\R^d} \p_t u_\eps \phi =  \int_0^T \int_{\R^d} \dv [ u_\eps \nabla \check{\rho}_\eps*{\rho}_\eps * u_\eps ] \phi,
\]
thus
\[
 \int_0^T \int_{\R^d}  u_\eps \p_t \phi =  \int_0^T \int_{\R^d}  u_\eps \nabla \check{\rho}_\eps*{\rho}_\eps * u_\eps  \Phi.
\]
where we have used the notation $\Phi=\nabla \phi$. The convergence is immediate for the time derivative but more delicate for the nonlinear drift term. We decompose: 
\[ \begin{aligned}
 \int_{ \R^d}  \Phi(t,x) u_\eps \nabla (\check{\rho}_\eps*\rho_\eps * u_\eps) &= \int_{ \R^d}  \rho_\eps* [\Phi(t,x) u_\eps] . \nabla (\rho_\eps * u_\eps)
\\
&=  \int_{ \R^d}  [\Phi (t,x) \,  \rho_\eps* u_\eps +r_\eps].  \nabla (\rho_\eps * u_\eps)
\end{aligned}
\]
 with \[
r_\eps (t,x)=  \int_{ \R^d}   [\Phi(t,y) - \Phi(t,x)] u_\eps(t,y) \rho_\eps (x-y) \diff y .
\]
 When we integrate in time, the first term passes to the limit because of strong-weak limits, and we want to prove that the second term vanishes. For $L_{ \Phi}$ the 1-Lipschitz norm of $\Phi,$ we write
\[
| r_\eps (t,x) | \leq L_\Phi \int_{ \R^d}   |x-y|  \rho_\eps (x-y) u_\eps(t,y) \diff y \leq \eps \; L_\Phi  \; \rho_\eps * u_\eps,
\]
since $\rho_\eps$ is compactly supported. We apply the Cauchy-Schwarz inequality
 \[ 
 \int_0^T \int_{\R^d} | r_\eps (t,x) | |\nabla (\rho_\eps * u_\eps)| \diff t \diff x \leq \eps \; L_\Phi  \; \| \rho_\eps * u_\eps \|_{L^2(0,T;L^2(\R^d))}  \| \nabla (\rho_\eps * u_\eps) \|_{L^2(0,T;L^2(\R^d))} \rightarrow 0.
\]

Hence at the limit $u_0$ is solution of \eqref{eq:limit}.

\section{The case of systems}\label{sec:multi}
%

 We may now extend the above proof  to a system of $N$ non-local equations \eqref{eq:Intro_multiple}, that is
\beq
\p_t u^{i}_\eps -  \dv [ u^{i}_\eps \sum_{j=1}^N \nabla K_\eps^{ij}* u^{j}_\eps ]=0 , \qquad t \geq 0 , \, x \in \R^d,  \, \qquad i \in [\![1,N]\!],
\label{sys:convol}
\eeq 
where $K_\eps^{ij}$ are kernel functions which satisfy
\[
K_\eps^{ij} (x) = \frac{1}{\eps^d} K^{ij} (\f x \eps), \qquad \gamma^{ij} := \int_{\R^d} K^{ij}(x) \diff x = \int_{\R^d} K_\eps^{ij}(x) \diff x < +\infty .
\] 
This models the evolution of a multi-species population, where each species $i$ can sense another species~$j$ through the non-local operator $K^{ij}$. The aim of this section is to establish that its limit  as $ \eps \rightarrow 0 $ is the system of porous media equations
\beq
\p_t u^{i}_0 -  \sum_{j=1}^N\dv [\gamma^{ij} u^{i}_0 \nabla u^{j}_0 ]=0 , \qquad t \geq 0 , \, x \in \R^d, \, \qquad i \in [\![1,N]\!].
\label{sys:convol_limit}
\eeq
This limit has been studied in the case $N=2$ in \cite{burger2022porous} using Wasserstein metrics related methods, under the assumption that $ \min \{\gamma^{11},\gamma^{22}\} > \frac{\gamma^{12}+\gamma^{21}}{12} \geq 0$ ; in~\cite{david2023degenerate}, the limit is proved with $n=2$ and $\gamma^{ij}=1$ for $i,j=1,2$. In this work the method used to prove the convergence is very different from the work in \cite{burger2022porous} as well as our hypothesis.

We now list a set of assumptions on the interaction function and on the initial data. For the initial data we assume that, with uniform bounds with respect to $\eps$, 
\beq 
{u_\eps^i}^0  \geq 0, \qquad |x|^2  {u_\eps^i}^0  \in L^1(\R^d) \qquad {u_\eps^i}^0  \in L^1(\R^d)\cap L^2(\R^d), \qquad{u_\eps^i}^0 \ln({u_\eps^i}^0) \in L^1(\R^d).
 \label{sys:as1}
\eeq
For the interaction function, we consider kernels  $K^{ij}$ that can be decomposed thanks to a matrix $ A:=(\alpha^{kj}) \in \mathcal{M}_{p,N}(\R)$ with $p \geq N$ and $(\rho_\eps^i)_{i \in [\![1,N]\!]} \in (L^1(\R^d))^N$ under the form
\beq \label{sys:as2}
K^{ij}(x) = \sum_{k=1}^p \alpha^{ki} \alpha^{kj} \check{\rho}^{i} * {\rho}^{j} (x), \qquad \forall (i,j) \in [\![1,N]\!]^2,
\eeq
with the notation $\check{\rho}^{i}(.) = {\rho}^{i}(-.)$. Additionally, we assume that
\beq 
\mbox{rank}(A) = N,\qquad \int_{\R^d} \rho^i(x)\diff x =1, \qquad    \rho^i(x)  \geq 0,  \qquad   \int_{\R^d} |x|^\f{d+2}{2}\rho^i(x)  \diff x < +\infty.
 \label{sys:as3}
\eeq
Under these assumptions, the matrix $K^{ij}(x)$ is not necessarily symmetric, but we remark that $\gamma^{ij} =\sum^p_{k=1} \alpha^{ki} \alpha^{kj}=(A^T A)_{ij}$ and that the matrix $G:=(\gamma^{ij})_{i,j}$ is symmetric positive definite. 

It is then natural to define
$$ \rho_\eps^i(x) = \f{1}{\eps^d}  \rho^i(\f{x}{\eps}) $$
which gives 
$K_\eps^{ij} = \sum_{k=1}^p \alpha^{ki} \alpha^{kj} \check{\rho}_\eps^{i} * {\rho}_\eps^{j} $. With these conditions, we may write 
\beq \begin{aligned} \label{sys:calcul_convol}
\sum_{i,j=1}^N  \int_{ \R^d}  f^{i}_\eps \; K_\eps^{ij}*f^{j}_\eps \diff x   &= \sum_{i,j=1}^N \sum_{k=1}^p \int_{ \R^d}  \alpha^{ki} \alpha^{kj} f^{i}_\eps \;  \check{\rho}^i_\eps* {\rho}^j_\eps*f^{j}_\eps \diff x 
\\
&= \sum_{k=1}^p  \int_{ \R^d} \Big(  \sum_{i=1}^N \alpha^{ki}  \rho^i_\eps*f^{i}_\eps \Big) \Big(  \sum_{j=1}^N \alpha^{kj}  \rho^j_\eps*f^{j}_\eps \Big)
\\
&= \sum_{k=1}^p \int_{ \R^d} \Big(  \sum_{i=1}^N \alpha^{ki}  \rho^i_\eps*f^{i}_\eps \Big)^2 ,
\end{aligned}
\eeq
and for $k\in[\![1,p]\!]$ we have
\[
\sum_{i=1}^N \alpha^{ki}  \rho^i_\eps*f^{i}_\eps  =  \big(A (\rho_\eps^i*f^{i}_\eps)_i  \big)_k \quad \mbox{ (as a matrix-vector product)}.
\]
We will often use the following computation: with the notation $A^{-1} = (\beta_{ij})$ the left inverse of $A$, we may write
\beq \label{sys:calcul_inv}
 \rho^j_\eps*f^{j}_\eps = \sum_{k=1}^p \beta_{jk} \Big( \sum_{i=1}^N \alpha^{ki}  \rho^i_\eps*f^{i}_\eps  \Big),
\eeq
hence
\[ 
  \| \rho^j_\eps*f^{j}_\eps \|_{L^2(\R^d)} \leq \sum_{k=1}^p |\beta_{jk}| \: \| \sum_{i=1}^N \alpha^{ki}  \rho^i_\eps*f^{i}_\eps \|_{L^2(\R^d)} .
\]

We are ready to state our main theorem.

\begin{theorem}
For $\ep\in (0,1),$ let $(\rho_\ep)_{i\in[\![1,N]\!]}$ be defined by~\eqref{sys:as2}-\eqref{sys:as3} and $({u^i_\ep}^0)_{i\in[\![1,N]\!]}$ be a family of  initial data  satisfying \eqref{sys:as1}. Let $({u^i_\ep})_{i\in[\![1,N]\!]}$ be a weak solution of \eqref{sys:convol}, then for all $i\in[\![1,N]\!]$ there exist subsequences $({u^i_\ep})$ and $(\rho_\ep^i *{u^i_\ep})$ that are not relabelled such that 
\begin{align}
\rho^{i}_\eps * u^{i}_\eps & \rightarrow u^{i}_0 \quad \mbox{ strongly in } \quad L^p(0,T;L^
q_{\mbox{\scriptsize{loc}}}(\R^d)) \quad \forall 1\leq p < 2, \; 1\leq q < \f{2d}{d-2},
 \label{sys:lim_L1strong}\\ 
 \rho^{i}_\eps * u^{i}_\eps & \rightharpoonup  u^{i}_0 \quad \mbox{ weakly in } \quad L^2(0,T;H^1(\R^d)), \label{sys:lim_H1weak} \\
 u^{i}_\eps & \rightharpoonup  u^{i}_0 \quad \mbox{ weakly in } \quad L^1_{\mbox{\scriptsize{loc}}}((0,T)\times\R^d). \label{sys:lim_Cinfweaku}
\end{align}
and $(u^i_0)_{i\in [\![1,N]\!]}$ is a solution of \eqref{sys:convol_limit} in the distributional sense.
\end{theorem}

This result is a generalisation of the result presented in Sec.~1 with two new features: the multi-species aspect of the population, and the fact that the interaction functions $\rho_\eps^i$ do not necessarily have a compact support. The proof is structured as follows. In Sec.~\ref{ss2.1} we present a priori estimates. In Sec.~\ref{ss2.2} we show compactness of $(\rho^{i}_\eps * u^{i}_\eps)_\eps$, and finally in Sec.~\ref{ss2.3} we show the convergence in Eq.~\eqref{sys:convol}.

\begin{remark} An example of interaction kernel. The Gaussian function used in \cite{JungelPortischZurek_2022} 
\beq 
K^{ij}_\eps(x)= \frac{\gamma^{ij}}{(2\pi\eps^2)^{d/2}} \exp \Big( -\frac{|x|^2}{2\eps^2} \Big)
\eeq
satisfies assumptions \eqref{sys:as2}-\eqref{sys:as3}. Here,  $\gamma^{ij}$ is defined such that there exists $A=(\alpha^{ij})$ invertible such that $\gamma^{ij} = \sum_{k=1}^p \alpha^{ki} \alpha^{ kj}$.
We remark that 
\[ \begin{aligned}
K^{ij}_\eps(x)  &= \ \gamma^{ij} \frac{e^{ -\frac{|x|^2}{2\eps^2} }}{(2\pi\eps^2)^{d/2}} = \gamma^{ij}  \int_{\R^d} \frac{e^{ -\frac{|x-y|^2}{\eps^2} }}{(\pi\eps^2)^{d/2}}  \frac{e^{ -\frac{|y|^2}{\eps^2} }}{(\pi\eps^2)^{d/2}} \diff y
\\
&= \sum_{k=1}^p \alpha^{ki} \alpha^{kj}  \check{\rho}^{i}_\eps * \rho^{j}_\eps (x),
\end{aligned}
\]
with $\rho^i_\eps(x)=\frac{e^{ -\frac{|x|^2}{\eps^2} }}{(\pi\eps^2)^{d/2}}$. Note than in this case $\rho_\eps^i = \check{\rho}_\eps^i$ and $\rho_\eps^i$  does not depend on $i$.
\end{remark}

\subsection{A priori estimates} \label{ss2.1}

\begin{proposition} \label{prop:ape}
Assume \eqref{sys:as1}-\eqref{sys:as3} and let $(u^i_\eps)_{i\in[\![1,N]\!]}$ be solutions of \eqref{sys:convol}. Then, for all $t \geq 0$ and $i \in [\![1,N]\!]$, $u^i_\eps \geq 0$ and, for some constants $C_1(T)$, $C_2(T)$ independent of $\ep$
$$ 
\int_{\R^d} |x|^2 u^{i}_\eps (t) \leq C_1 (T) \qquad \mbox{and} \qquad \int_{\R^d} u^{i}_\eps(t) \,|\ln \; u^{i}_\eps (t)|  \leq C_2 (T) , \qquad \forall\; t\in [0,T].
$$
Moreover, we have the following estimates, uniform with respect to $\ep$,
\begin{align}
& u^i_\eps \in L^\infty(0,T;L^1(\R^d)) \label{sys:uL1}, \\
& \rho^i_\eps*u^{i}_\eps \in L^\infty(0,T;L^1(\R^d)) \label{sys:L1}, \\ 
& \rho^i_\eps*u^{i}_\eps \in L^2(0,T;H^1(\R^d))   \label{sys:H1} ,\\ 
& \rho^i_\eps*u^{i}_\eps \in L^2(0,T;L^{\f{2d}{d-2}}(\R^d))  \label{sys:L2d/d-2},  \qquad {\text{for }}d\geq 3, \\ 
& \sqrt{u^{i}_\eps}  \sum_{j=1}^N \nabla K_\eps^{ij}*u^{j}_\eps  \in L^2((0,T) \times\R^d ). \label{sys:entropy1}
\end{align}
For $d=2,$ the estimates are the same except~\eqref{sys:L2d/d-2} which is replaced by $\rho^i_\eps*u^{i}_\eps \in L^2(0,T;L^{p}(\R^d))$ for any $1\leq p<\infty,$ see also Remarks~\ref{remark4} and~\ref{remark5} of the one species case.

\end{proposition}

\begin{proof}
\noindent {\em Step 1. Positivity, $L^1$ bound and $L^2$ bound.} 
Let us first prove that $u^i_\eps \geq 0$ for $i \in [\![1,N]\!]$. We multiply Eq.~\eqref{sys:convol} by $-\1_{u^i_\ep \geq 0}$ to get
\[\p_t |u^i_\ep |_- = \dv [|u^i_\ep |_- \nabla_x (K_\eps^{ij} * u^i_\ep)] .
\]
Integrating in space gives
\[ \f{\diff }{\diff t} \dst \int_{\R^d} |u^i_\ep |_- \diff x = 0.
\]
Hence $|u^i_\ep(t,\cdot)|_-=|{u_\eps^i}^0|_-=0$ for all $i \in [\![1,N]\!]$, which gives us the positivity. Additionally, for a given $i \in [\![1,N]\!]$, integrating Eq.~\eqref{sys:convol} in space gives
\[ \frac{\diff}{\diff t} \int_{\R^d} u^i_\ep (t,x) \diff x= \int_{\R^d} \sum_{j=1}^N\dv [ u^{i}_\eps \nabla K_\eps^{ij}* u^{j}_\eps ]\diff x = 0. \]
Therefore $ \int_{\R^d} u^i_\ep (t,x) \diff x=\int_{\R^d} {u_\eps^i}^0 \diff x $ and  we have estimate \eqref{sys:uL1}. Since $\rho^i \in L^1(\R^d)$, we deduce easily the estimate \eqref{sys:L1}. The $L^2$ control comes from the following computation: on the one hand we have
\[ \begin{aligned}
\f 12 \f{\diff }{\diff t}  \sum_{i,j=1}^N \int_{ \R^d} u^{i}_\eps  \; K_\eps^{ij}*u^{j}_\eps \diff x &= \f 12  \sum_{i,j=1}^N \int_{ \R^d}  \partial_t u^{i}_\eps  \; K_\eps^{ij}*u^{j}_\eps \diff x +  \f 12  \sum_{i,j=1}^N \int_{ \R^d} u^{i}_\eps  \; K_\eps^{ij}* \partial_tu^{j}_\eps \diff x
\\ 
& = \sum_{i,j=1}^N \int_{ \R^d}  \partial_t u^{i}_\eps  \; K_\eps^{ij}*u^{j}_\eps \diff x
\\
& = - \sum_{i=1}^N  \int_{ \R^d} u^{i}_\eps  |\sum_{j=1}^N \nabla K_\eps^{ij}*u^{j}_\eps|^2 \diff x \leq 0.
\end{aligned}
\]
On the other hand, given the formula~\eqref{sys:calcul_convol}, we have
\[  
\sum_{i,j=1}^N \int_{ \R^d} u^{i}_\eps  \; K_\eps^{ij}*u^{j}_\eps = \sum_{k=1}^p \int_{ \R^d} \Big(  \sum_{i=1}^N \alpha^{ki}  \rho^i_\eps*u^{i}_\eps \Big)^2 .
\] 
Thus, inserting it in the previous equation, and after integrating in time, we find
\[   \sum_{k=1}^p \int_{ \R^d} \Big(  \sum_{i=1}^N \alpha^{ki}  \rho^i_\eps*u^{i}_\eps \Big)^2+ \sum_{i=1}^N  \int_{ 0}^T \int_{ \R^d}  u^{i}_\eps  |\sum_{j=1}^N \nabla K_\eps^{ij}*u^{j}_\eps|^2 \leq \sum_{k=1}^p \int_{ \R^d} \Big(  \sum_{i=1}^N \alpha^{ki}  \rho^i_\eps*{u_\eps^i}^0 \Big)^2.
\]
This immediately proves \eqref{sys:entropy1} and we also have that $ \sum_{i=1}^N \alpha^{ki}  \rho^i_\eps*u^{i}_\eps \in L^\infty(0,T;L^2(\R^d ))$ for all $k =1, ... ,p $. 
Moreover, given Eq.~\eqref{sys:calcul_inv} we deduce 
\[
\rho^i_\eps*u^{i}_\eps = \sum_{k=1}^p \beta_{ik} \Big( \sum_{i=1}^N \alpha^{ki}  \rho^i_\eps* u^{i}_\eps  \Big) \in L^\infty(0,T;L^2(\R^d )) \quad \forall i =1, ... ,N.
\]
\\[5pt]
\noindent {\em Step 2. Second moment control.} 
For a given $i\in [\![1,N]\!]$, we compute
\[ \begin{aligned}
 \f{\diff }{\diff t}  \int_{\R^d}| x|^2 u^i_\eps  & =  \int_{\R^d} |x|^2 \dv [u^i_\ep   ( \sum_{j=1}^N \nabla_x K_\eps^{ij} * u^i_\ep)] \\
 & = - \int_{\R^d} 2x u^i_\ep ( \sum_{j=1}^N \nabla_x  K_\eps^{ij} * u^i_\ep) \\
 & \leq  2 \Big(\int_{\R^d} |x|^2 u^i_\ep \Big)^{1/2} \Big( \int_{\R^d} u^i_\ep \big| ( \sum_{j=1}^N \nabla_x K_\eps^{ij} * u^i_\ep) \big|^2\Big)^{1/2} .
\end{aligned}
\]
Hence 
\[
 \f{\diff }{\diff t} \Big( \int_{\R^d} |x|^2 u^i_\eps \Big )^{1/2} \leq \Big( \int_{\R^d} u^i_\ep \big|  \sum_{j=1}^N \nabla_x K_\eps^{ij} * u^i_\ep \big|^2\Big)^{1/2} .
\]
finally, after integration in time we find
\[\ \Big( \int_{\R^d} |x|^2 u^i_\eps \Big )^{1/2} \leq \Big( \int_{\R^d} |x|^2 {u_\eps^i}^0 \Big )^{1/2} + \int_0^t \Big( \int_{\R^d} u^i_\ep \big|  \sum_{j=1}^N \nabla_x K_\eps^{ij} * u^i_\ep \big|^2\Big)^{1/2} .
\]
Applying the Cauchy-Schwarz inequality to the second term of the right-hand side of the equation and taking the square, we obtain the annouced estimate 
\[\  \int_{\R^d}| x|^2 u^i_\eps \leq   2\int_{\R^d} |x|^2 {u^i_\eps}^0 + 2t \int_0^t \int_{\R^d} u^i_\ep \big| ( \sum_{j=1}^N \nabla_x K_\eps^{ij} * u^i_\ep) \big|^2  < +\infty,
\]
thanks to estimate \eqref{sys:entropy1} and the initial condition \eqref{sys:as1}.
\\[5pt]
{\em Step 3. Space compactness.}  We consider the classical entropy $\sum_{i=1}^N \int_{ \R^d} u^{i}_\eps \ln (u^{i}_\eps)$ and compute its derivative

\[\begin{aligned}
\f{\diff }{\diff t}  \sum_{i=1}^N \int_{ \R^d} u^{i}_\eps \ln (u^{i}_\eps)&=- \sum_{i,j=1}^N  \int_{ \R^d}  \nabla u^{i}_\eps \; K_\eps^{ij}*\nabla u^{j}_\eps \diff x
\\
& = - \sum_{k=1}^p \int_{ \R^d} \Big(  \sum_{i=1}^N \alpha^{ki}   \nabla \rho^i_\eps*u^{i}_\eps \Big)^2 \diff x,
\end{aligned} \]
where we have used~\eqref{sys:calcul_convol}. Then, after integrating in time we get
\[  \sum_{i=1}^N \int_{ \R^d} u^{i}_\eps |\ln (u^{i}_\eps)| +  \sum_{k=1}^p  \int_0^T \int_{ \R^d} \Big(  \sum_{i=1}^N \alpha^{ki}   \nabla \rho_\eps*u^{i}_\eps \Big)^2 \diff x \diff t \leq \sum_{i=1}^N \int_{ \R^d} {u^{i}_\eps}^0 \ln ({u^{i}_\eps}^0) +2 \sum_{i=1}^N \int_{ \R^d} u^{i}_\eps |\ln (u^{i}_\eps)|_-.
\]
Moreover for a given $i\in [\![1,N]\!]$, $\int_{ \R^d} u^{i}_\eps |\ln (u^{i}_\eps)|_- \diff x$ can be decomposed as follows:
\[\dst\int_{ \R^d} u^{i}_\eps |\ln (u^{i}_\eps)|_- \diff x=\dst \int_{ \R^d} u^{i}_\eps |\ln (u^{i}_\eps)|_- \1_{u^{i}_\ep \geq e^{-|x|^2}} \diff x + \dst\int_{ \R^d} u_\eps |\ln (u^{i}_\eps)|_- \1_{u^{i}_\ep \leq e^{-|x|^2}}\diff x.
\]
We then bound each term, noticing first that when $u_\ep \geq e^{-|x|^2}$ then $ |\ln (u^{i}_\eps)|_- =-\ln (u^{i}_\eps) \1_{u^{i}_\ep \leq 1} \leq |x|^2$
\[
\int_{ \R^d} u^{i}_\eps |\ln (u^{i}_\eps)|_- \1_{u^{i}_\ep \geq e^{-|x|^2}} \diff x =\dst\int_{ \R^d} u^{i}_\eps |\ln (u^{i}_\eps)|_- \1_{u^{i}_\ep \geq e^{-|x|^2}} \diff x
\leq 
\dst\int_{ \R^d} |x|^2 u^{i}_\eps \diff x < +\infty,
 \]
 and, noticing now that when $x\geq 1$ and $y \leq e^{-|x|^2}$ then $ y |\ln (y)| \leq e^{-|x|^2} |\ln (e^{-|x|^2})|$ because $x\mapsto x|\ln(x)|$ is increasing on $(0,e^{-1})$ (with a maximum on $e^{-1}$)
 \begin{multline*}
   \dst\int_{ \R^d} u_\eps |\ln (u^{i}_\eps)|_- \1_{u^{i}_\ep \leq e^{-|x|^2}}\diff x =\dst \int_{ \R^d} u^{i}_\eps |\ln (u^{i}_\eps)|_- \1_{u^{i}_\ep  \leq e^{-|x|^2}}\diff x
  \\ \leq \dst \int_{|x|\leq 1} u^{i}_\eps |\ln (u^{i}_\eps)|_- \1_{u^{i}_\ep \leq 1}\diff x
  +
  \dst \int_{ |x|\geq 1} |x|^2 e^{-|x|^2} \diff x < + \infty.  
 \end{multline*} 
Then $ \sum_{i=1}^N \alpha^{ki}  \nabla \rho^{i}_\eps*u^{i}_\eps \in  L^2(0,T;L^2(\R^d)) $ for all $k =1, ... ,p $. Using again the inversion formula \eqref{sys:calcul_inv}, we deduce 
\[
\nabla \rho^{j}_\eps*u^{j}_\eps \in L^2(0,T; L^2(\R^d)) \quad \forall j =1, ... ,N ,
\]
which gives the estimates \eqref{sys:H1} (since we have already proved the corresponding $L^2$ bound). In addition using the  Gargliano-Nirenberg-Sobolev inequality, see \cite{LiebLoss}, we have
\[
\| \rho^{j}_\eps*u^{j}_\eps \|_{L^2(0,T;L^{\f{2d}{d-2}}(\R^d))} \leq \| \nabla \rho^{j}_\eps*u^{j}_\eps \|_{L^2(0,T;L^2(\R^d))},
\]
for $d\geq 3$, which shows estimate \eqref{sys:L2d/d-2}. This concludes the proof of Proposition~\ref{prop:ape}.
\end{proof}

\subsection{Compactness} \label{ss2.2}

Our next step is to prove space and time compactness. We first remark that, thanks to the bounds \eqref{sys:L1} and \eqref{sys:H1}, we have
\beq \label{sys:L1loc}
\rho^{i}_\eps * u^{i}_\eps \in L^\infty(0,T;L^1_{\mbox{\scriptsize{loc}}}(\R^d)), \quad \nabla \rho^{i}_\eps * u^{i}_\eps \in L^1_{\mbox{\scriptsize{loc}} }((0,\infty)\times\R^d),  \quad \forall i \in [\![1,N]\!], 
\eeq
which provides us with space compactness. To pass to the limit $\eps \rightarrow 0$ we need to obtain some compactness in time. To this aim, we compute the equation of $\rho^i_\eps * u^{i}_\eps$ for a given $i\in [\![1,N]\!]$,
\[
\p_t \rho^{i}_\eps * u^{i}_\eps = \dv [\underbrace{ \rho^{i}_\eps* \Big( u^i_\eps  \sum_{j=1}^N  \; \nabla_x K_\eps^{ij}*u^{j}_\eps \Big)}_{:=Q^i_\eps(t,x)}],
\]
with $Q^i_\eps(t,x)$ defined by $
Q^i_\eps(t,x) = \int_{\R^d} \rho^i_\eps (x-y) u^i_\eps (y)   \sum_{j=1}^N  \; \nabla_x K_\eps^{ij} *u^{j}_\eps(y) \diff y.$
Thus, we may write
\[\begin{aligned}
\|  Q^i_\eps(t,x)  \|_{L^1(\R^d)}  &\leq  \|  \rho^i_\eps  \|_{L^1(\R^d)}  \int_{\R^d} \   u^i_\eps (y)  \sum_{j=1}^N  \; \nabla_x K_\eps^{ij} *u^{j}_\eps(y) \diff y
\\ &
 \leq   \f{1}{2} \Big(   \|  \sqrt{u^i_\eps}  \|^2_{L^2(\R^d)} +  \|  \sqrt{u^i_\eps}  \sum_{j=1}^N  \; \nabla_x K^{ij}_\eps *u^{j}_\eps  \|^2_{L^2(\R^d)} \Big),
\end{aligned}
 \]
 and
\beq   \|  Q^i_\eps(t,x)  \|_{L^1(0,T;L^1(\R^d))} \leq \f{1}{2}  \Big(  T\|  u^i_\eps  \|_{L^\infty(0,T;L^1(\R^d))} + \|  \sqrt{u^i_\eps}   \sum_{j=1}^N  \; \nabla_x K^{ij}_\eps *u^{j}_\eps  \|_{L^2(0,T;L^2(\R^d))}^2 \Big).
\eeq
Therefore $Q^i_\eps$ is bounded in $ L^1((0,T)\times\R^d)$. This, together with~\eqref{sys:L1loc}, allows us to show that for a compact set $K \subset \R^d$ and $k>0$, we have
\beq \label{sys:Comp_t}
\Vert \rho^i_\ep * u^i_\ep (t+k,x) - \rho^i_\ep * u^i_\ep (t,x) \Vert_{L^1((0,T)\times K)} \leq C(K,T) k \Vert \nabla u^i_\ep * \rho^i_\ep \Vert_{L^2((0,T)\times K)} \leq C(K,T) \sqrt{k}.
 \eeq
 Indeed considering a mollifier sequence $\omega_\eta(x) = \frac{1}{\eta^d} \omega_(\frac{1}{\eta})$, we have

\[ \begin{aligned} \Vert \rho^i_\ep * u^i_\ep (t+k,x) - \rho^i_\ep * u^i_\ep (t,x) \Vert_{L^1((0,T)\times K)} \leq &  \Vert \rho^i_\ep * u^i_\ep (t+k,x) - \rho^i_\ep * u^i_\ep *\omega_\eta (t+k,x) \Vert_{L^1((0,T)\times K)} \\
& + \Vert \rho^i_\ep * u^i_\ep (t,x) - \rho^i_\ep * u^i_\ep *\omega_\eta (t,x) \Vert_{L^1((0,T)\times K)}  \\
& + \Vert \rho^i_\ep * u^i_\ep*\omega_\eta (t+k,x) - \rho^i_\ep * u^i_\ep *\omega_\eta (t,x) \Vert_{L^1((0,T)\times K)}  ,\\
 \end{aligned} \] 
 and
 \[ \begin{aligned} 
 & \Vert \rho^i_\ep * u^i_\ep (t+k,x) - \rho^i_\ep * u^i_\ep *\omega_\eta (t+k,x) \Vert_{L^1((0,T)\times K)} \\
 & \qquad = \int_0^T \int_K \int_K \Big( \rho^i_\ep * u^i_\ep (t+k,x) - \rho^i_\ep * u^i_\ep (t+k,x-y) \Big) \omega_\eta(y) \diff y  \diff x \diff t \\
& \qquad \leq C(K,T) \eta \int_0^T \int_K \int_K  \nabla \rho^i_\ep * u^i_\ep (t+k,x) \omega_\eta(y) \diff y  \diff x \diff t \\
& \qquad \leq \tilde{C}(K,T) \eta \Vert \nabla u_\ep * \rho_\ep \Vert_{L^2((0,T)\times K)} .
 \end{aligned}\]
Similarly, we have
 \[
    \Vert \rho^i_\ep * u^i_\ep (t,x) - \rho^i_\ep * u^i_\ep *\omega_\eta (t,x) \Vert_{L^1((0,T)\times K)}  \leq \tilde{C}(K,T) \eta \Vert \nabla u_\ep * \rho_\ep \Vert_{L^2((0,T)\times K)}
\]
 Finally, we obtain
 \[ \begin{aligned} 
\Vert \rho^i_\ep * u^i_\ep*\omega_\eta (t+k,x) - \rho^i_\ep * u^i_\ep *\omega_\eta (t,x) \Vert_{L^1((0,T)\times K)} & = \Vert \int_t^{t+k} \p_t \rho^i_\ep * u^i_\ep*\omega_\eta (s,x) \diff s\Vert_{L^1((0,T)\times K)} \\
& \leq \Vert \int_t^{t+k} \dv (Q^i_\ep) *\omega_\eta (s,x) \diff s\Vert_{L^1((0,T)\times K)} \\
& \leq \int_t^{t+k}  \Vert Q^i_\ep *\dv (\omega_\eta) (s,x) \Vert_{L^1((0,T)\times K)}  \diff s \\
& \leq C(K) \f{k}{\eta}  \Vert Q^i_\ep  \Vert_{L^1((0,T)\times K)} .
 \end{aligned}\]
 With the choice $\eta = \sqrt{k}$ we have the result.

Then, given \eqref{sys:L1loc} and \eqref{sys:Comp_t}, for all $i\in [\![1,N]\!]$, $\rho^i_\ep * u^i_\ep$ satisfies the assumptions of the Weil-Kolmogorov-Frechet theorem on $L^1((0,T)\times K),$ hence the sequence $(\rho^i_\ep * u^i_\ep)$ is compact in this space. 

\subsection{Convergence} \label{ss2.3}

Given the previous estimates, we can extract a subsequence (we do not relabel it for the sake of simplicity) such that
\begin{align}
\rho^{i}_\eps * u^{i}_\eps & \rightarrow u^{i}_0 \quad \mbox{ strongly in } \quad L^1(0,T;L^1_{\mbox{\scriptsize{loc}}}(\R^d)),
 \label{sys:lim_L1strongbis}\\ 
 \rho^{i}_\eps * u^{i}_\eps & \rightharpoonup  u^{i}_0 \quad \mbox{ weakly in } \quad L^2(0,T;H^1(\R^d)), \label{sys:lim_H1weakbis} \\
 u^{i}_\eps & \rightharpoonup  u^{i}_0 \quad \mbox{ weakly in } \quad L^1_{\mbox{\scriptsize{loc}}}((0,T)\times\R^d). \label{sys:lim_Cinfweakubis}
\end{align}

 In addition, from the uniform bound  $\rho^i_\eps*u^{i}_\eps \in L^2(0,T;L^{\f{2d}{d-2}}(\R^d))$ thanks to estimates \eqref{sys:L2d/d-2}, and the strong limit \eqref{sys:lim_L1strongbis}, we have
 \beq \label{sys:lim_Lpstrong}
 \rho^{i}_\eps * u^{i}_\eps \rightarrow u^{i}_0 \quad \mbox{ strongly in } \quad L^p(0,T;L^
q_{\mbox{\scriptsize{loc}}}(\R^d)) \quad \forall 1\leq p < 2, \; 1\leq q < \f{2d}{d-2}.
 \eeq

 We are left to show the convergence in Eq.~\eqref{sys:convol}. Let us take $\phi^i \in C_c^\infty([0,T]\times\R^d)$ for $i =1, ... ,N $. Then multiplying Eq.~\eqref{sys:convol} by $\phi^i$ and integrating by parts, we have 
\begin{align*}
 \int_0^T \int_{\R^d} \p_t u^{i}_\eps \phi^i &=  \int_0^T \int_{\R^d} \dv [ u^{i}_\eps  \sum_{j=1}^N \nabla_x K_\eps^{ij}* u^{j}_\eps ] \phi^i
\\
&=- \int_0^T \int_{\R^d}  u^{i}_\eps ( \sum_{j=1}^N \nabla_x K_\eps^{ij}* u^{j}_\eps ) \Phi^i
 \end{align*}
where we have used the notation  $\Phi^i=\nabla \phi^i$. For the term on the right-hand side, the weak limit \eqref{sys:lim_Cinfweakubis} shows that
 \[  \int_0^T \int_{\R^d} u^{i}_\eps \p_t \phi^i \rightarrow \int_0^T \int_{\R^d} u^{i}_0 \p_t \phi^i.  \]
The term on the right-hand side can be rewritten as
\[ \begin{aligned}
 \int_{ \R^d}  u^i_\eps \nabla ( \sum_{j=1}^N K_\eps^{ij} * u^j_\eps)  \Phi^i &= \sum_{k=1}^p \sum_{j=1}^N \alpha^{ki} \alpha^{kj} \int_{ \R^d}  \Phi^i(t,x) u^i_\eps . \nabla (\check{\rho}^{i}_\eps * \rho^{j}_\eps * u^j_\eps) 
\\
&= \sum_{k=1}^p \sum_{j=1}^N \alpha^{ki} \alpha^{kj}  \int_{ \R^d}  \rho^{i}_\eps* [\Phi^i(t,x) u^i_\eps] . \nabla (\rho^{j}_\eps * u^j_\eps)
\\
&= \sum_{k=1}^p \sum_{j=1}^N \alpha^{ki} \alpha^{kj} \int_{ \R^d}  [\Phi^i (t,x) \,  \rho^{i}_\eps* u^i_\eps +r^{i}_\eps].  \nabla (\rho^{j}_\eps * u^j_\eps),
\end{aligned}
\]
with $$r^{i}_\eps (t,x) =  \int_{ \R^d}   [\Phi^i(t,y) - \Phi^i(t,x)] u^i_\eps(t,y) \rho^i_\eps (x-y) \diff y  .$$
Then, we obtain a decomposition in two terms
\[
\int_{ \R^d}  u^i_\eps \nabla ( \sum_{j=1}^N K_\eps^{ij} * u^j_\eps)  \Phi^i = \sum_{k=1}^p \sum_{j=1}^N \alpha^{ki} \alpha^{kj}  \int_{ \R^d}    \rho^{i}_\eps* u^i_\eps \, \Phi^i (t,x) .  \nabla (\rho^{j}_\eps * u^j_\eps) + \sum_{k=1}^p \sum_{j=1}^N \alpha^{ki} \alpha^{kj} \int_{ \R^d}  r^{i}_\eps \ .  \nabla (\rho^{j}_\eps * u^j_\eps).
\]

The first term passes to the limit because of strong-weak limits. Indeed for $j,k \in [\![1,N]\!]$, we have successively
\begin{align}
\alpha^{ki} \rho^{i}_\eps* u^i_\eps & \rightarrow \alpha^{ki} u^i_0 \quad \mbox{strongly in } L^p(0,T;L^
q_{\mbox{\scriptsize{loc}}}(\R^d)) \quad \forall 1\leq p < 2, \; 1\leq q < \f{2d}{d-2}, 
\\
\alpha^{ki} \nabla \rho^{j}_\eps* u^j_\eps & \rightharpoonup  \alpha^{kj} \nabla u^j_0  \quad \mbox{weakly in } L^2((0,T)\times \R^d) ,
\end{align}
and thus
\begin{multline*}
  \sum_{k=1}^p \sum_{j=1}^N \alpha^{ki} \alpha^{kj} \int_{ \R^d}    \rho^{i}_\eps* u^i_\eps \, \Phi^i (t,x) .  \nabla (\rho^{j}_\eps * u^j_\eps) \\ \rightarrow \sum_{k=1}^p \sum_{j=1}^N \alpha^{ki} \alpha^{kj} \int_{ \R^d}    u^i_0 \, \Phi^i (t,x) .  \nabla  u^j_0 = \int_{ \R^d}    u^i_0 \, \Phi^i (t,x) . \sum_{j=1}^N  \gamma^{ij} \nabla  u^j_0   
\end{multline*}

For the second term, since $\nabla (\rho^{j}_\eps * u^j_\eps) \in L^2((0,T)\times \R^d)$, we need to prove that $r^{ki}_\eps$ converges to $0$ in $L^2((0,T)\times\R^d)$. With $L^i_\phi$ the Lipschitz constant for the function $\Phi^i$, we have
\[  \begin{aligned}
r^{i}_\eps (t,x) & \leq L^i_\phi   \int_{ \R^d}  | x-y | u^i_\eps(t,y) | \rho^i_\eps (x-y)|\diff y  \\
  & \leq L^i_\phi   \int_{ \R^d}  | x-y | u^i_\eps(t,y)  \rho^i_\eps (x-y) \diff y  \\
 & \leq L^i_\phi    \int_{ \R^d}  | x-y | \Big( u^i_\eps(t,y) | \rho^i_\eps (x-y)| \Big)^{\frac{d}{d+2}+ \frac{2}{d+2}} \diff y \\
 & \leq L^i_\phi  \Big[ \int_{ \R^d}  u^i_\eps(t,y)  \rho^i_\eps (x-y) \diff y \Big]^{\frac{d}{d+2}} \Big[ \int_{ \R^d} |x-y|^\frac{d+2}{2} u^i_\eps(t,y)  \rho^i_\eps (x-y) \diff y\Big]^{\frac{2}{d+2}} \\
 & \leq L^i_\phi  \Big[u^i_\eps(t) * \rho^i_\eps \Big]^{\frac{d}{d+2}} \Big[ \int_{ \R^d} |x-y|^\frac{d+2}{2} u^i_\eps(t,y)  \rho^i_\eps (x-y) \diff y\Big]^{\frac{2}{d+2}}.
\end{aligned}\]
Next, we use the H\" older inequality with $p=\frac{d+2}{d-2}$, $p'=\frac{d+2}4$ (notice however that the case $d=2$ is slighly different and left to the reader, then we choose $p <\infty$ and $p'>1$, but close to, and the same arguments work) and we write
\[ \begin{aligned}
\|r^{i}_\eps\|^2_{L^2(\R^d)} &\leq  {L^i_\phi}^2 \int_{ \R^d}  \Big[u^i_\eps(t,\cdot) * \rho^i_\eps \Big]^{\frac{2d}{d+2}} \Big[ \int_{ \R^d} |x-y|^\frac{d+2}{2} u^i_\eps(t,y)  \rho^i_\eps (x-y) \diff y\Big]^{\frac{4}{d+2}} \diff x 
\\
&\leq {L^i_\phi}^2 \| u^i_\eps(t,\cdot) * \rho^i_\eps \|^{\frac{2d}{d+2}}_{L^\frac{2d}{d-2}{  (\R^d)}}
\Big[ \int_{\R^{2d}} |x-y|^\frac{d+2}{2} u^i_\eps(t,y)  \rho^i_\eps (x-y) \diff y \diff x \Big]^{\frac{4}{d+2}}  
\\
&\leq {L^i_\phi}^2 \| u^i_\eps(t,\cdot) * \rho^i_\eps \|^{\frac{2d}{d+2}}_{L^\frac{2d}{d-2}{ (\R^d)}} \| u^i_\eps(t,y) \|_{L^1{ (\R^d)}}^\frac{4}{d+2} \eps^2 \big[\int_{\R^{d}} |z|^\frac{d+2}{2} \rho^i (z) \diff z\big]^\frac{4}{d+2},
\end{aligned} \]
where we have used the change of variable $x \to z= \frac{x-y}{\eps}$. Since we have proved that $ u^i_\eps(t) * \rho_\eps$ is uniformly bounded in $L^2(0,T;L^{\f{2d}{d-2}}(\R^d))$, we conclude that 
\[ \begin{aligned}
\| r^{i}_\eps\|^2_{L^2(0,T;L^2(\R^d)}  & \leq \eps^{2} {L^i_\phi}^2  
\big \|u^i_\eps \big\|^\frac{4}{d+2}_{L^\infty(0,T;L^1(\R^d))} 
\big \| z^\frac{d+2}{2} \rho  \big\|^\frac{4}{d+2}_{L^1(\R^d)}
\int_0^T  \big\| u^i_\eps(t,y) * \rho^i_\eps \big \|^{\frac{2d}{d+2}}_{L^\f{2d}{d-2}(\R^d)} \diff t \\
 & \leq \eps^{2} {L^i_\phi}^2  
\big \|u^i_\eps \big\|^\frac{4}{d+2}_{L^\infty(0,T;L^1(\R^d))} 
\big \| z^\frac{d+2}{2} \rho^i  \big\|^\frac{4}{d+2}_{L^1(\R^d)}
T^{\frac 2{d+2}} \Big(\int_0^T  \big\| u^i_\eps(t,y) * \rho^i_\eps \big \|^2_{L^\f{2d}{d-2}(\R^d)} \diff t \Big)^{\frac{d}{d+2}}
\end{aligned} \]

Thus
 \begin{multline*}
   \sum_{k=1}^p \sum_{j=1}^N \alpha^{ki} \alpha^{kj} \int_0^T \int_{\R^d} r^{ki}_\eps (t,x) \nabla (\rho^{kj}_\eps * u^j_\eps) \\ \leq \big( \sup_{(i,j)\in[\![1,N]\!]^2} | \alpha^{ki}|\big)^2 \sum_{k=1}^p \sum_{j=1}^N    \| r^{ki}_\eps\|^2_{L^2(0,T;L^2(\R^d))}  \| \nabla (\rho^{kj}_\eps * u^j_\eps) \|^2_{L^2(0,T;L^2(\R^d))}, \rightarrow 0.  
 \end{multline*}
 
Therefore, we have proved that at the limit the equation holds
\[ 
\int_0^T \int_{\R^d} u^{i}_0 \p_t \phi^i - \sum_{j,k=1}^N \alpha^{ki} \alpha^{kj}\int_0^T \int_{\R^d}  u^{i}_0 \nabla ( \sum_{j=1}^N u^{j}_0 ) \nabla\phi^i=0,
 \]
and thus $u^{i}_0 $ is solution of \eqref{sys:convol_limit}.

\section{Numerical simulations} \label{num}
We illustrate the localisation limit in 2D thanks to numerical simulations.  We consider a spatial square domain $\Omega = [x_{\min}, x_{\max}] \times [y_{\min}, y_{\max}]\in \mathbb{R}^2$ taken large enough so that the solution (spatial) support stays far from the boundaries. The spatial domain is discretized into $N_x \times N_x$ regularly spaced points with space step $\Delta x$ and we consider four species. The numerical schemes used to solve the nonlocal model \eqref{sys:convol} and the limiting model \eqref{sys:convol_limit} are adapted from \cite{carrillo_chertock_huang_2015,Barre2018}. 

We consider interaction kernels of the form \eqref{sys:as2} where the $\rho^i(x)$ for $i=1\hdots 4$ are chosen as Gaussian functions of variance $\sigma_i^2$ :
$$
\rho^i(x) = \frac{1}{2\pi \sigma_i^2} \exp^{- \frac{|x|^2}{2 \sigma_i^2}}.
$$
With this choice, $K^{ij}(x)$ is again a Gaussian function of variance $\sigma_i^2 + \sigma_j^2$. Therefore, $K^{ij}$ are of the form:
$$
K^{ij}(x) =  \frac{\big( A^T A \big)_{ij}}{2\pi (\sigma_i^2+\sigma_j^2)} \exp^{- \frac{|x|^2}{2 (\sigma_i^2+\sigma_j^2)}}.
$$
Note that thus defined, the coefficients $\alpha^{ij}$ of the matrix $A$ control the intensities of the interactions, while the variances $\sigma_i^2$ control the interaction distance between the different species. We consider three cases:
\begin{itemize}
    \item case 1: Species-dependent repulsion intensities: $(A^TA)_{ij} = \min(i,j)$ (the matrix A is an upper triangular matrix with coefficient equal to 1) and same interaction distance for all species, i.e., $\sigma_i = \sigma_j \quad \forall (i,j) =1\hdots 4$. These conditions satisfy hypothesis~\eqref{sys:as3}.       
    \item case 2: Species-dependent repulsion intensities: $(A^TA)_{ij} = \min(i,j)$ (the matrix A is an upper triangular matrix) and species-dependent interaction distances, i.e., $\sigma_i \neq \sigma_j \forall i\neq j$. These conditions satisfy hypothesis \eqref{sys:as3}.
    \item case 3: Same repulsion intensities for all species: $(A^T A)_{ij} = 1$ (matrix A is a matrix of rank 1) and species-dependent interaction distances, i.e., $\sigma_i \neq \sigma_j, \;  \forall i\neq j$. These conditions violate the full rank condition for matrix A, therefore we do not satisfy hypothesis \eqref{sys:as3}.
\end{itemize}

All species are initially supposed to be uniformly distributed in a ball of radius $S = 2$ centered in the center of the domain:
$$
\rho_i(x) = \frac{1}{\pi S^2} \mathds{1}_{B(0,S)}(x), \quad \forall \; i=1 \hdots 4,
$$
and we let the system evolve until time $t=100$. We aim to study numerically the convergence of the nonlocal system \eqref{sys:convol} to the local system \eqref{sys:convol_limit} as $\varepsilon \rightarrow 0$ in the different cases for the choice of the interaction kernel. 

\subsection{Case 1: same interaction distances for all species $\sigma_i^2 = \sigma_j^2 = 0.1$, different interaction intensities $(A^TA)_{ij} = \min(i,j)$}

We show in Fig. \ref{figure:cas1_spatial} the spatial distributions at time $t=100$ obtained with the nonlocal model \eqref{sys:convol} for different values of $\varepsilon$ ($\varepsilon = 5$, first column, $\varepsilon = 1$, second column and $\varepsilon = 0.5$, third column) and with the local model \eqref{sys:convol_limit} (fourth column) for each of the 4 species (different lines). 

\begin{figure}[h]
    \centering
    \includegraphics[scale = 0.23]{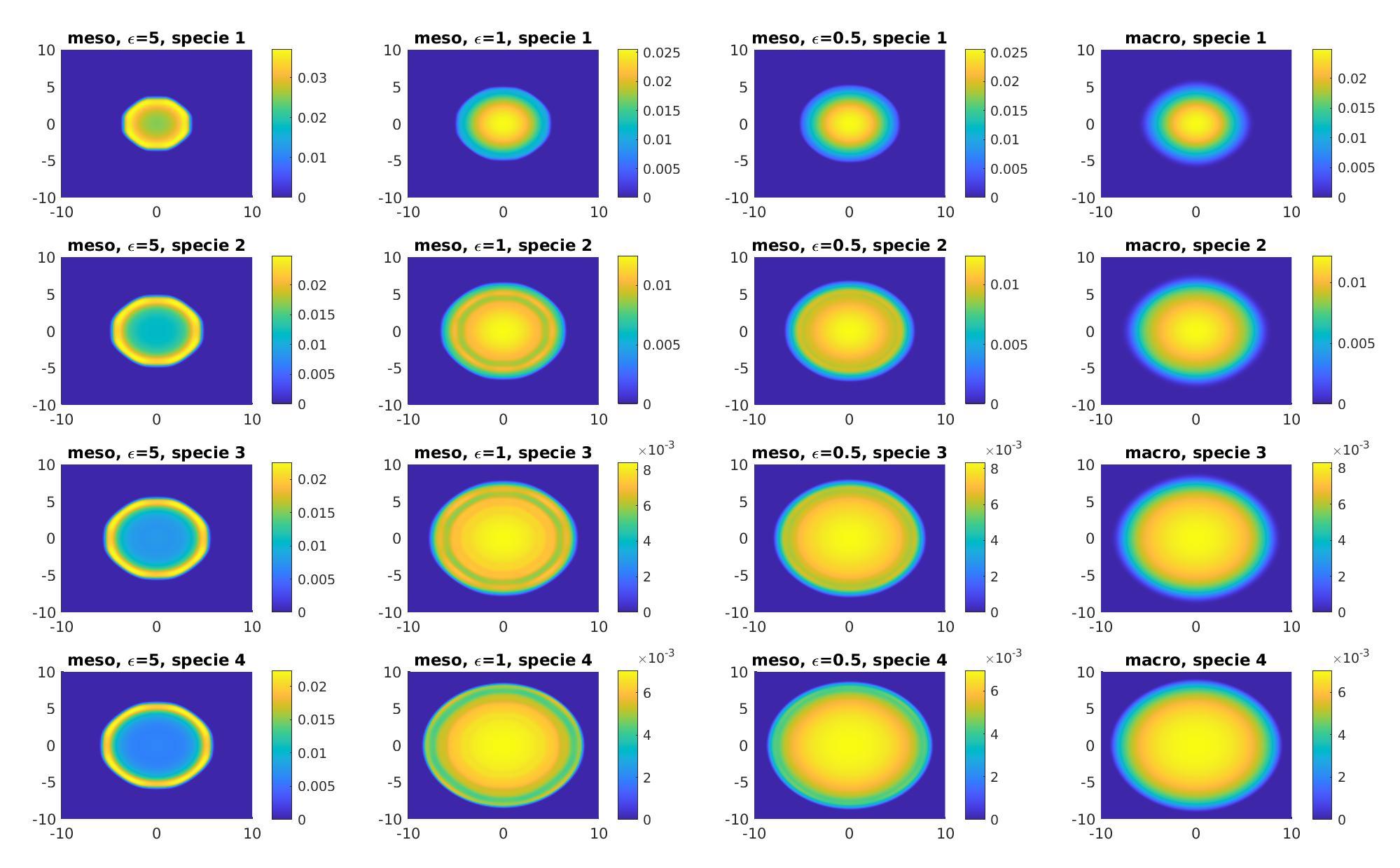}
    \caption{Case 1: Solutions of the nonlocal model \eqref{sys:convol} for different values of $\varepsilon$ ($\varepsilon = 5$, first column, $\varepsilon = 1$, second column and $\varepsilon = 0.5$, third column) and of the local model \eqref{sys:convol_limit} (fourth column) at time $t=100$, for each of the 4 species (different lines). } 
    \label{figure:cas1_spatial}
\end{figure}

As one can observe in Fig. \ref{figure:cas1_spatial} and as expected, the solution evolves radially in space. Moreover, for this choice of interaction kernel, we observe that the species with larger indices diffuse faster than the first species (compare the support of the solution from top to bottom). Indeed, for interaction intensities of the form $(A^TA)_{ij} = \min(i,j)$, species with larger indices interact more strongly with themselves than species with smaller indices, resulting in stronger repulsion. 
Comparing the columns from left to right, we observe that the nonlocal model is quite far from the local model for $\varepsilon = 5$ (first column), but seems to get closer as $\varepsilon$ decreases. For large $\varepsilon$, we observe the concentration of the solution in rings close to the boundary of the support due to the Gaussian form of the interaction, while reducing $\varepsilon$ diminishes this effect and the nonlocal model gets closer to the local one.

Since the solutions are spreading radially, an efficient way to characterize the
dynamics is to compute the radial distribution $g(\lambda)$, which gives the average density on rings of size $\diff\lambda$:
$$
g^\varepsilon_i(\lambda,t)\diff\lambda = \int_{B(0, \lambda + \diff\lambda) \backslash B(0,\lambda)} u^\varepsilon_i(x,t) \diff x,
$$
and we denote by $g^0_i(x,t)$ the quantity obtained in the same way with the solution of the local model $u^0_i(x,t)$. In the left subplots of Fig. \ref{figure:cas1_radial}, we show the quantities $g^0_i(\lambda,100)$ (plain lines) $g^\varepsilon_i(\lambda,100)$ (dotted lines) for different values of $\varepsilon$ (different colors) and for  each of the 4 species (different subplots). We again observe the concentration of the solutions of the nonlocal model in rings for large $\varepsilon$ (red curves), while the radial distributions of the nonlocal model get closer to the one of the local model when decreasing $\varepsilon$ (compare the green and blue lines). In order to quantify the convergence of the non-local to the local model as $\varepsilon \rightarrow 0$, we plot in Fig. \ref{figure:cas1_radial} the discrete $L^2$ distance between $u^\varepsilon$ and $u^0$ at time 100 as function of $\varepsilon$ for each of the 4 species.  
\begin{figure}[h]
    \centering
    \includegraphics[scale = 0.23]{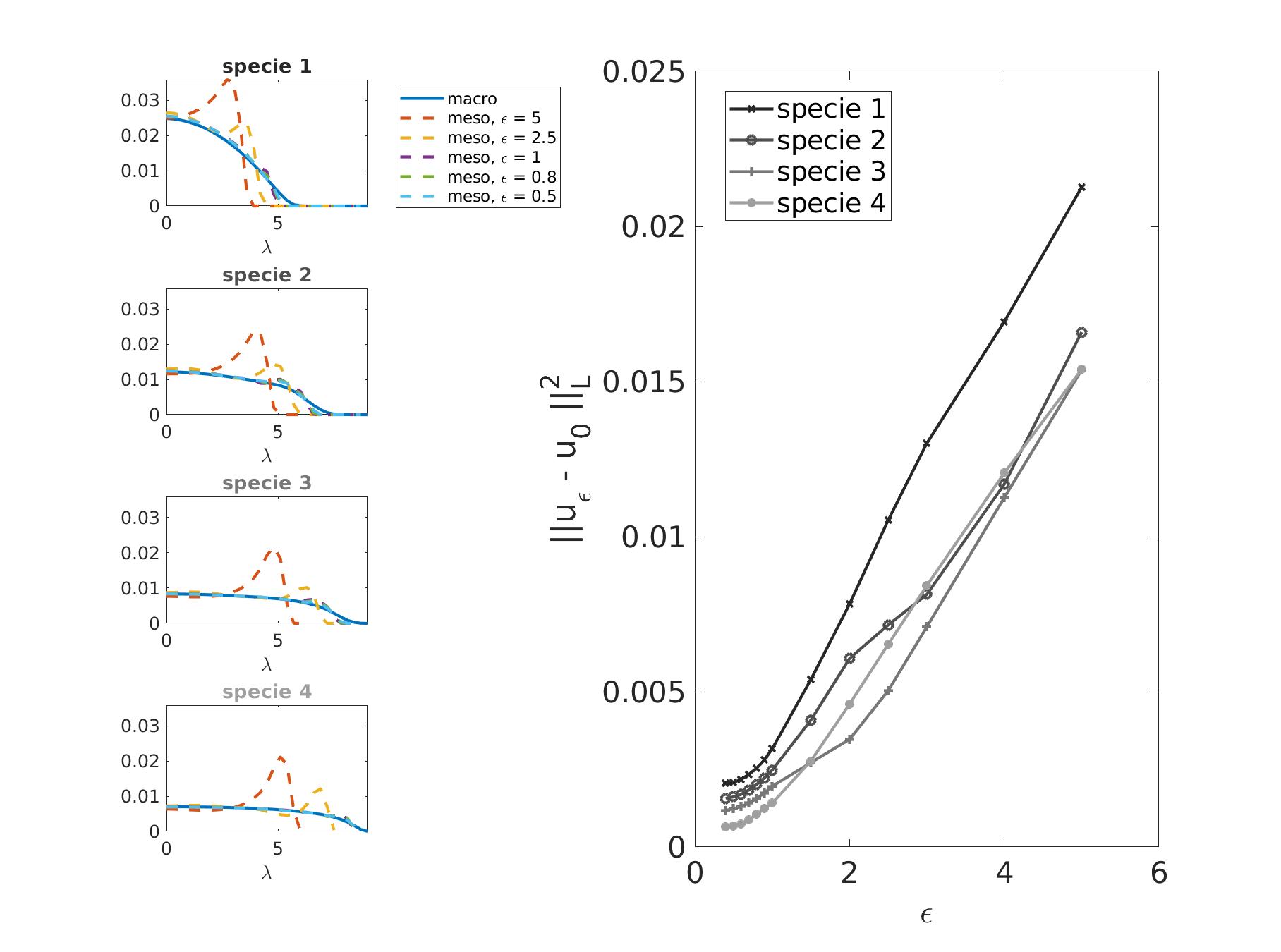}
    \vspace{-25pt}
    \caption{Case 1: Solutions of the nonlocal model \eqref{sys:convol} for different values of $\varepsilon$ ($\varepsilon = 5$, first column, $\varepsilon = 1$, second column and $\varepsilon = 0.5$, third column) and of the local model \eqref{sys:convol_limit} (fourth column) at time $t=100$, for each of the 4 species (different lines). } 
    \label{figure:cas1_radial}
\end{figure}

As one can see, the nonlocal model seems to converge to the local one as $\varepsilon$ decreases for all species, with a faster convergence rate for species~1 than for larger indices. Again, the ring effect is visible for large values of $\varepsilon$.

\subsection{Case 2: different interaction distances as function of the species $\sigma_i \neq \sigma_j$, different interaction intensities $(A^TA)_{ij} = \min(i,j)$}

We show in Fig.~\ref{figure:cas2_spatial} the simulations obtained in Case~2 with $\sigma_1^2 = 0.1, \sigma_2^2 = 0.2, \sigma_3^2 = 0.3, \sigma_4^2 = 0.4$ (increasing interaction distance with increasing species indices). We adopt the same representation as in Fig.~\ref{figure:cas1_spatial}. It is noteworthy that the local model obtained as the limit $\varepsilon \rightarrow 0$ of the nonlocal one is the same here as for Case~1. 

\begin{figure}[h]
    \centering
    \includegraphics[scale = 0.23]{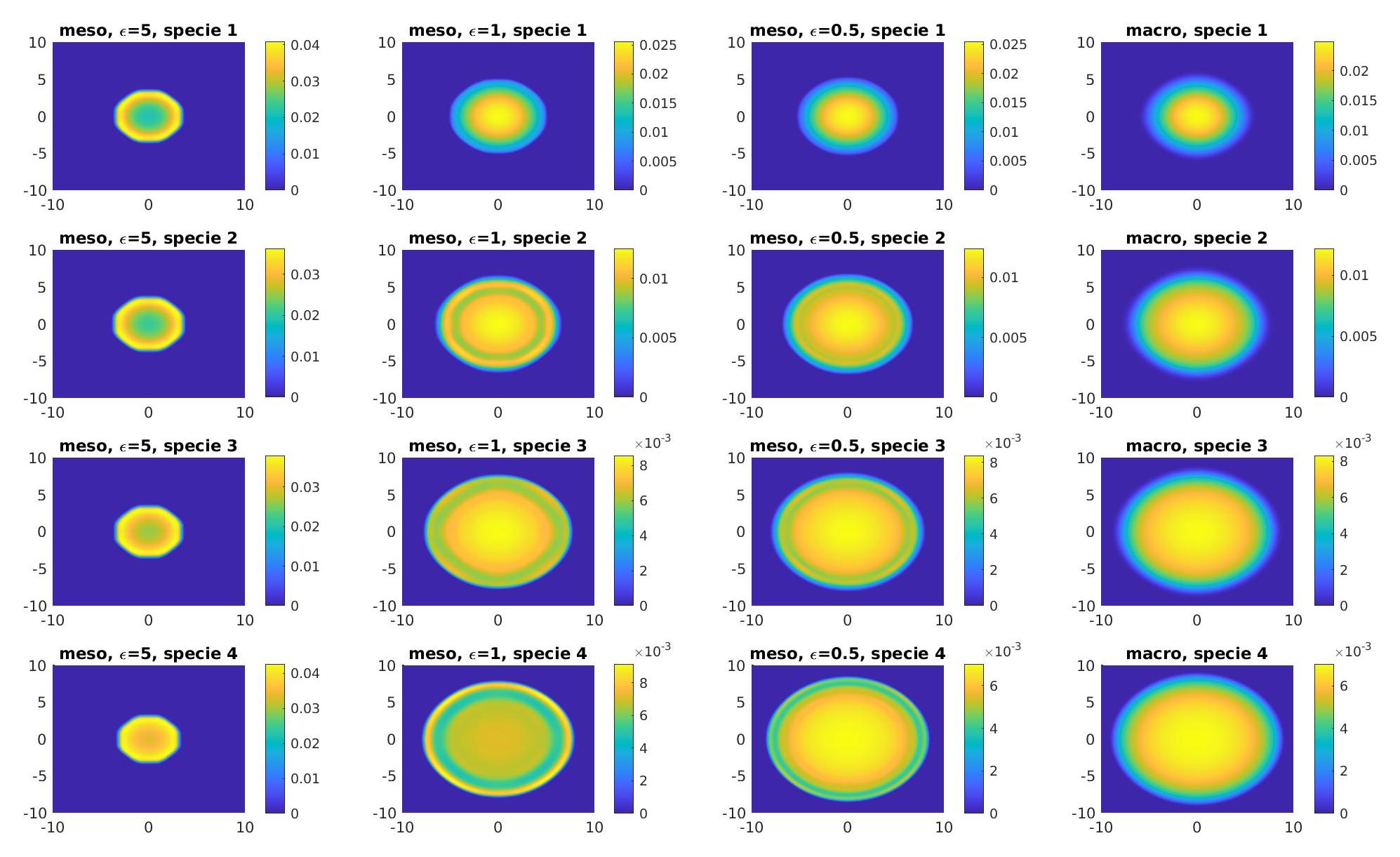}
    \caption{Case 2: Solutions of the nonlocal model \eqref{sys:convol} for different values of $\varepsilon$ ($\varepsilon = 5$, first column, $\varepsilon = 1$, second column and $\varepsilon = 0.5$, third column) and of the local model \eqref{sys:convol_limit} (fourth column) at time $t=100$, for each of the 4 species (different lines). } 
    \label{figure:cas2_spatial}
\end{figure}

As one can observe in Fig.~\ref{figure:cas2_spatial}, for $\varepsilon = 5$, a quite different behavior is obtained when the interaction distance is species-dependent compared to when the species interact at the same distance (compare the first column of this figure with the one of Fig.~\ref{figure:cas1_spatial}). In this case $\varepsilon = 5$ indeed, species with larger indices do not diffuse faster than species with lower indices as observed previously, although the coefficients $\gamma^{ij}$ are the same. This is due to the fact that the pairs interacting the stronger (large species indices) are also the ones interacting the farer, reducing de facto the interaction intensity (because of the choice of the Gaussian forms for $\rho^i$). As a result, all species diffuse less than in the previous case, but this effect is damped when localizing the interaction (i.e., decreasing $\varepsilon$). For small values of $\varepsilon = 0.5$ indeed, we recover a profile close to the one obtained with the local model (compare the last two columns of Fig. \ref{figure:cas2_spatial}). 

\begin{figure}[h]
    \centering
    \includegraphics[scale = 0.23]{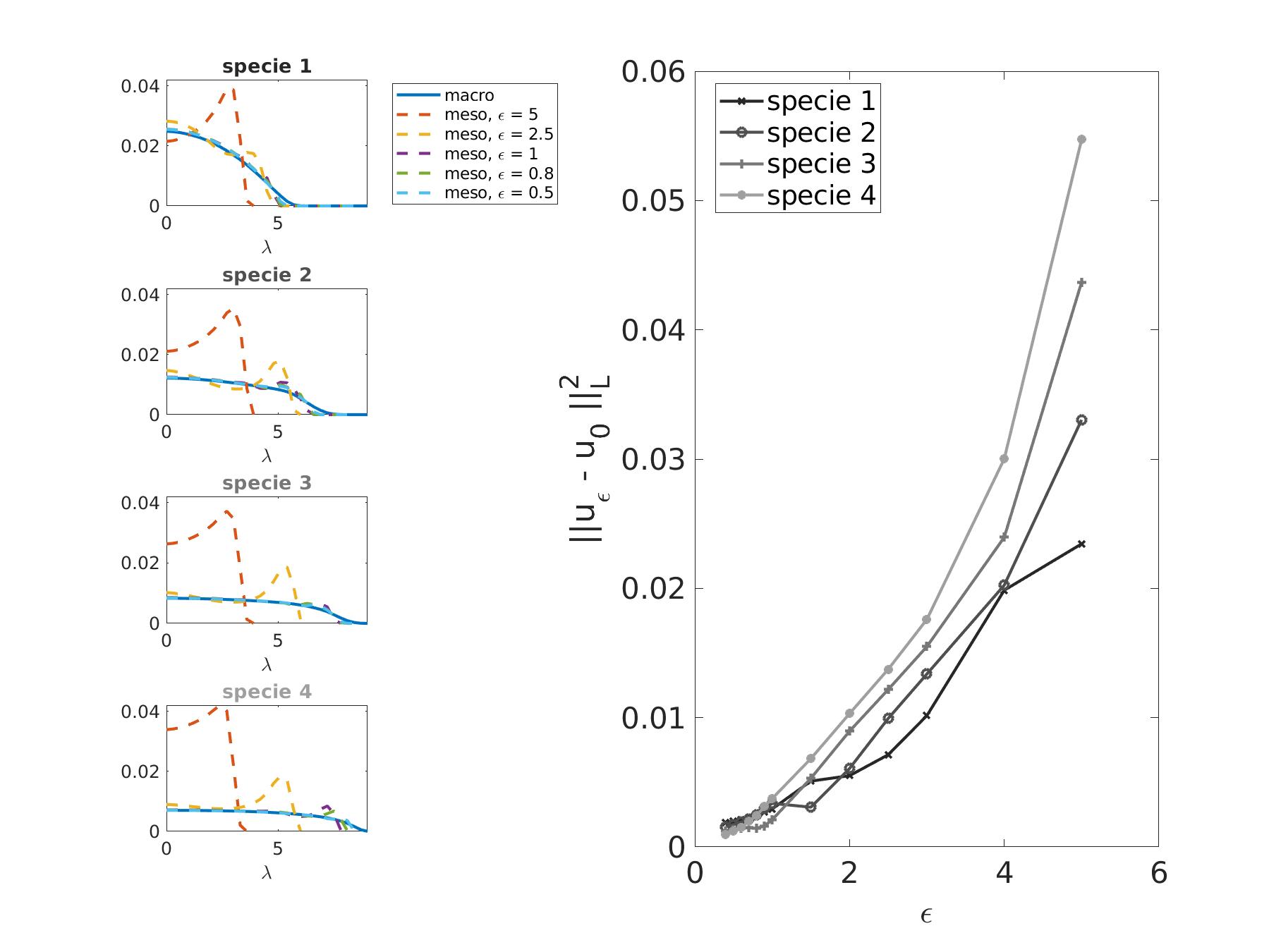}
    \vspace{-25pt}
    \caption{Case 2: Solutions of the nonlocal model \eqref{sys:convol} for different values of $\varepsilon$ ($\varepsilon = 5$, first column, $\varepsilon = 1$, second column and $\varepsilon = 0.5$, third column) and of the local model \eqref{sys:convol_limit} (fourth column) at time $t=100$, for each of the 4 species (different lines). } 
    \label{figure:cas2_radial}
\end{figure}

Regarding the convergence of the nonlocal to local model (Fig. \ref{figure:cas2_radial}), we observe that contrary to Case~1, species~4 seems to converge faster to the local model as $\varepsilon \rightarrow 0$ than the other species. This is because species~4 is the one the most impacted by the interaction distance controlled by $\varepsilon$.

\subsection{Case 3: different interaction distances as function of the species $\sigma_i \neq \sigma_j$, same interaction intensities $(A^TA)_{ij} = 1, \;  \forall (i,j)$}
\label{sec:case3}

Finally, we aim to study the convergence of the nonlocal to the local model when the hypothesis \eqref{sys:as3} of our convergence theorem is not satisfied. Namely, we consider here the case where matrix $A$ is of rank~1, i.e., all species interact with the same intensity $(A^TA)_{ij} = 1 \forall (i,j)$.

As before, we show in Fig.~\ref{figure:cas3_spatial} the simulations obtained in Case~3 with $\sigma_1^2 = 0.1$, $\sigma_2^2 = 0.2$, $\sigma_3^2 = 0.3$, $\sigma_4^2 = 0.4$ (increasing interaction distance with increasing species index). We adopt the same representation as in Figs.~\ref{figure:cas1_spatial} and \ref{figure:cas2_spatial}. 

\begin{figure}[h]
    \centering
    \includegraphics[scale = 0.23]{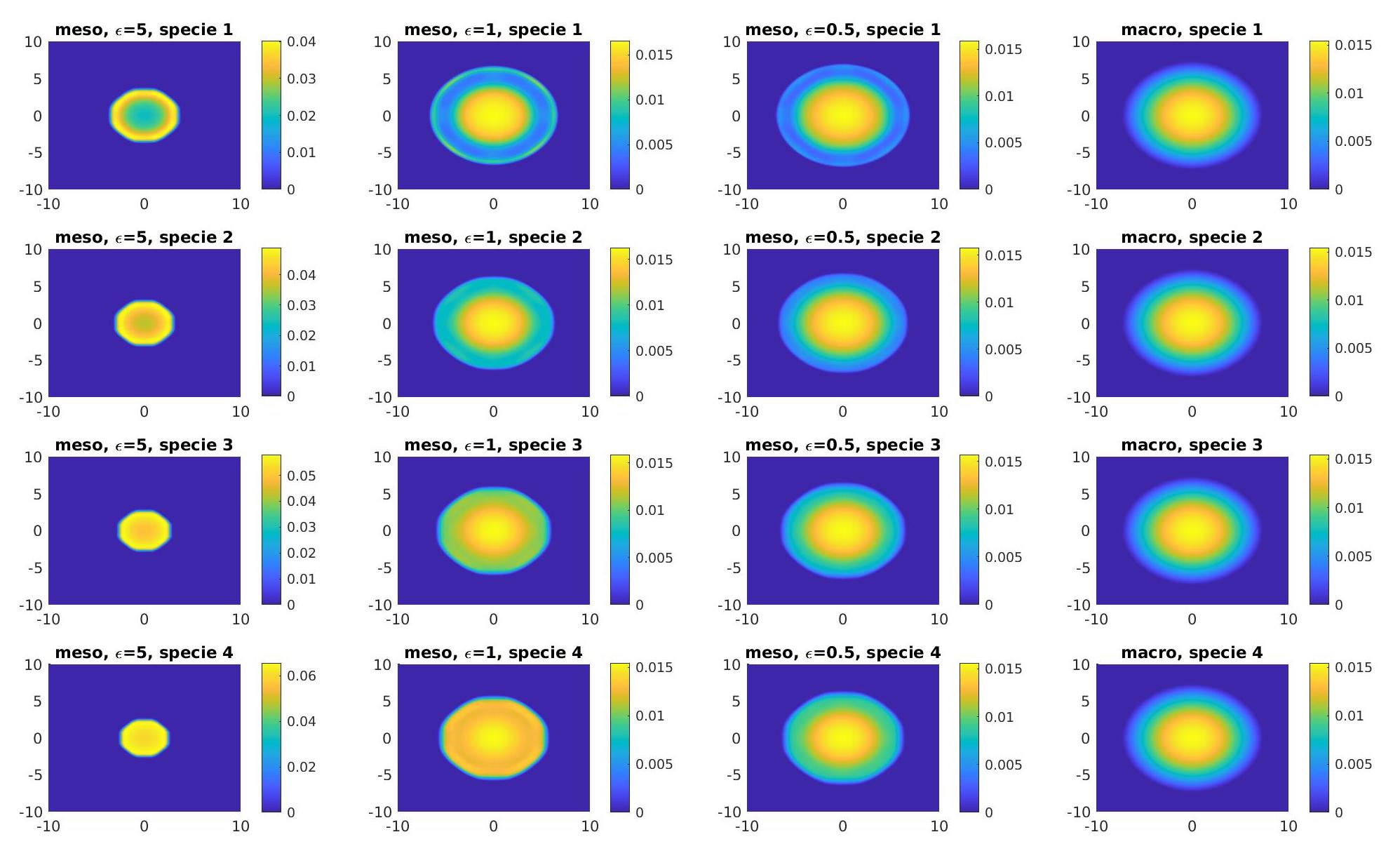}
    \caption{Case 3: Solutions of the nonlocal model \eqref{sys:convol} for different values of $\varepsilon$ ($\varepsilon = 5$, first column, $\varepsilon = 1$, second column and $\varepsilon = 0.5$, third column) and of the local model \eqref{sys:convol_limit} (fourth column) at time $t=100$, for each of the 4 species (different lines). } 
    \label{figure:cas3_spatial}
\end{figure}

In this case, the diffusion is even slower for species with large indices compared to Case~2 for $\varepsilon = 5$. It is noteworthy that for this case, all species behave in the same way in the localisation limit, as expected since all coefficients for the interactions are the same and the interaction distances do not play a role in the localisation limit. We still observe a quite good correspondence between the nonlocal and local model as $\varepsilon$ decreases, even though our interaction kernel does not satisfy the full rank hypothesis of $A$. 
These results are confirmed in Fig.~\ref{figure:cas3_radial}, where we observe that the error decreases for all species as $\varepsilon \rightarrow 0$. 

\begin{figure}[h]
    \centering
    \includegraphics[scale = 0.23]{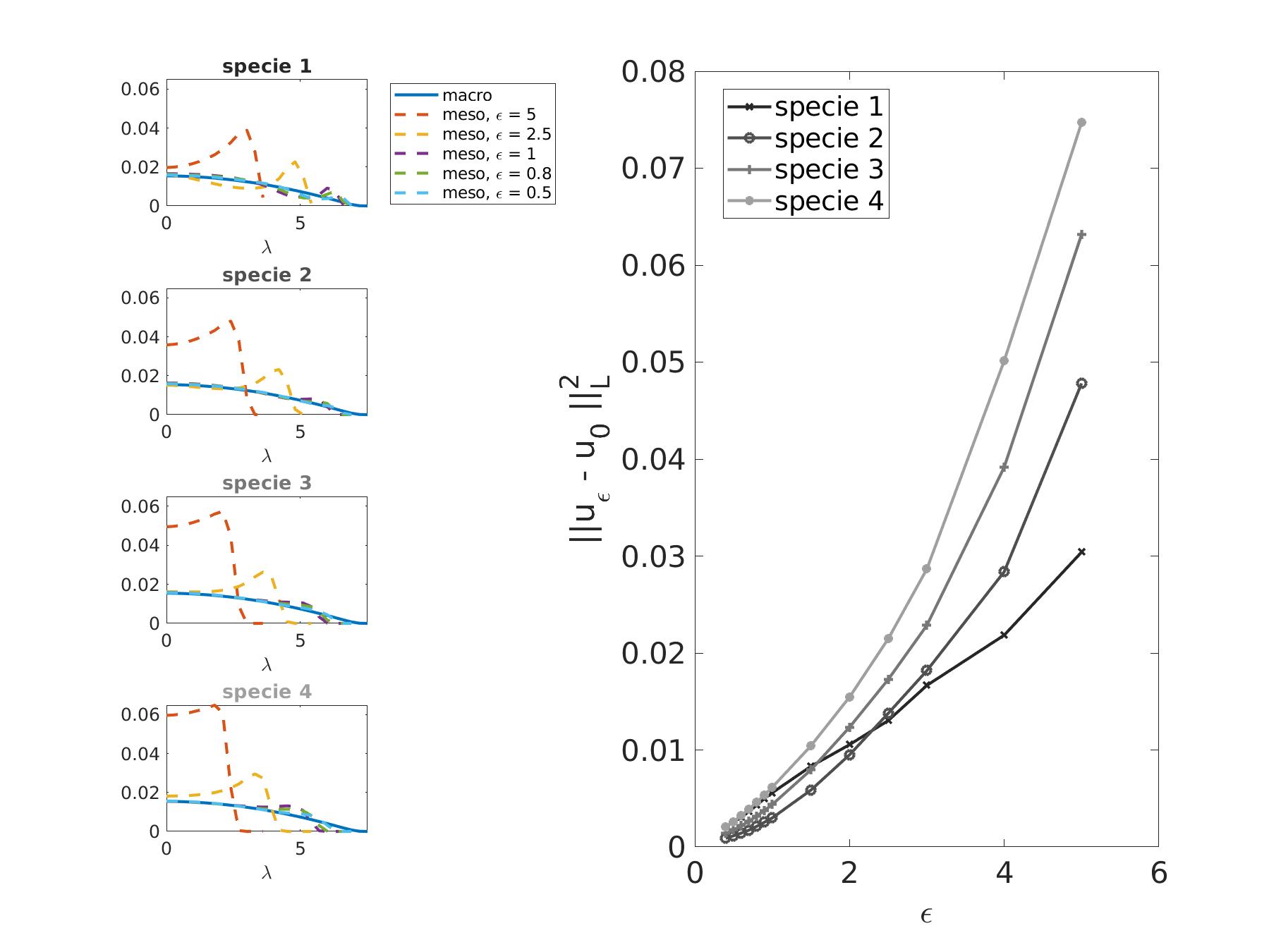}
    \vspace{-25pt}
    \caption{Case 3: Solutions of the nonlocal model \eqref{sys:convol} for different values of $\varepsilon$ ($\varepsilon = 5$, first column, $\varepsilon = 1$, second column and $\varepsilon = 0.5$, third column) and of the local model \eqref{sys:convol_limit} (fourth column) at time $t=100$, for each of the 4 species (different lines). } 
    \label{figure:cas3_radial}
\end{figure}

These numerical results suggest that the nonlocal model converges to the local one when $\varepsilon \rightarrow 0$ in a more general framework than the one considered in this paper, i.e., without the full rank condition for $A$.  

\section{Conclusion and perspectives}

We have established the convergence of solutions from non-local to local systems describing aggregation for both single and multi-species population, a subject that have attracted a lot of attention recently. The major new feature in our analysis is that we do not need diffusion to gain compactness, at odd with much of the existing literature including one of the most advanced result in~\cite{JungelPortischZurek_2022}. The central compactness result is provided by a full rank assumption on the interaction kernel. To remove this condition is a challenging question which is certainly possible, in some cases at least, in view  of the numerical simulations of Sec.~\ref{sec:case3}.

In turn, we prove existence of weak solutions for the resulting system, a cross-diffusion system of quadratic type, Eq.~\eqref{sys:convol}. This system is always symmetric due to the use of an energy (gradient flow) structure which is fundamental to provide the necessary a priori estimates. To extend the type of interaction is also an interesting question. 

Notice that the form of this system differs from a closely related class, namely
\[
\p_t u^{i}_0 -  \Delta \big[ \,  u^{i}_0 (1+\sum_{j=1}^N \gamma^{ij}  u^{j}_0 ) \, \big]=0 , \qquad t \geq 0 , \, x \in \R^d, \, \qquad i \in [\![1,N]\!],
\label{sys:convol_limit2}
\]
which is also an active field of research, \cite{Jungel_Cross2004,MoussaCD20} for which the nonlocal to local convergence is also studied.

\medskip

{\bf Acknowledgments.} DP was supported by Sorbonne Alliance University with an Emergence project MATHREGEN, grant number S29-05Z101 and by Agence Nationale de la Recherche (ANR) under the project grant number ANR-22-CE45-0024-01. The authors warmly thank Markus Schmidtchen for his careful reading and suggestions for the bibliography.
%
%
%
\bibliographystyle{plain}

\bibliography{mesomacro.bib}

\begin{thebibliography}{10}

\bibitem{ALASIO2022113064}
Luca Alasio, Maria Bruna, Simone Fagioli, and Simon Schulz.
\newblock Existence and regularity for a system of porous medium equations with
  small cross-diffusion and nonlocal drifts.
\newblock {\em Nonlinear Analysis}, 223:113064, 2022.

\bibitem{Barre2018}
Julien Barr\'e, José Carrillo, Pierre Degond, Diane Peurichard, and Ewelina
  Zatorska.
\newblock Particle interactions mediated by dynamical networks: Assessment of
  macroscopic descriptions.
\newblock {\em J Nonlinear Sci}, 28:235–268, 2018.

\bibitem{Barre2017}
Julien Barr\'{e}, Pierre Degond, and Ewelina Zatorska.
\newblock Kinetic theory of particle interactions mediated by dynamical
  networks.
\newblock {\em Multiscale Modeling \& Simulation}, 15(3):1294--1323, 2017.

\bibitem{Barre2020}
Julien Barré, Pierre Degond, Diane Peurichard, and Ewelina Zatorska.
\newblock Modelling pattern formation through differential repulsion.
\newblock {\em Networks and Heterogeneous Media}, 15(3):307--352, 2020.

\bibitem{Bertozzi2009}
Andrea~L Bertozzi, Jos{\'e}~A Carrillo, and Thomas Laurent.
\newblock Blow-up in multidimensional aggregation equations with mildly
  singular interaction kernels.
\newblock {\em Nonlinearity}, 22(3):683, 2009.

\bibitem{Bertozzi2007}
Andrea~L. Bertozzi and Thomas Laurent.
\newblock Finite-{T}ime {B}low-up of {S}olutions of an {A}ggregation {E}quation
  in {R}$^n$.
\newblock {\em Communications in Mathematical Physics}, 274(3):717--735, 2007.

\bibitem{Bertozzi2011}
Andrea~L. Bertozzi, Thomas Laurent, and Jes{\'u}s Rosado.
\newblock L$^p$ theory for the multidimensional aggregation equation.
\newblock {\em Communications on Pure and Applied Mathematics}, 64(1):45--83,
  2023/04/20 2011.

\bibitem{bruna2017diffusion}
Maria Bruna, S~Jonathan Chapman, and Martin Robinson.
\newblock Diffusion of particles with short-range interactions.
\newblock {\em SIAM Journal on Applied Mathematics}, 77(6):2294--2316, 2017.

\bibitem{burger2022porous}
Martin Burger and Antonio Esposito.
\newblock Porous medium equation and cross-diffusion systems as limit of
  nonlocal interaction, 2022.
\newblock arXiv 2202.05030.

\bibitem{Burger2008}
Martin Burger and Marco~Di Francesco.
\newblock Large time behavior of nonlocal aggregation models with nonlinear
  diffusion.
\newblock {\em Networks and Heterogeneous Media}, 3(4):749--785, 2008.

\bibitem{carrillo2023degenerate}
Jos{\'e}~A. Carrillo, Charles Elbar, and Jakub Skrzeczkowski.
\newblock Degenerate {C}ahn-{H}illiard systems: From nonlocal to local.
\newblock hal-04043503, 23pp, March 2023.

\bibitem{carrillo2013new}
Jos{\'e}~A Carrillo, Stephan Martin, and Vladislav Panferov.
\newblock A new interaction potential for swarming models.
\newblock {\em Physica D: Nonlinear Phenomena}, 260:112--126, 2013.

\bibitem{carrillo2019blob}
Jos{\'e}~Antonio Carrillo, Katy Craig, and Francesco~S Patacchini.
\newblock A blob method for diffusion.
\newblock {\em Calculus of Variations and Partial Differential Equations},
  58:1--53, 2019.

\bibitem{carrillo2020measure}
Jos{\'e}~Antonio Carrillo, Marco Di~Francesco, Antonio Esposito, Simone
  Fagioli, and Markus Schmidtchen.
\newblock Measure solutions to a system of continuity equations driven by
  newtonian nonlocal interactions.
\newblock {\em Discrete and Continuous Dynamical Systems}, 40(2):1191--1231,
  2020.

\bibitem{carrillo2023nonlocal}
Jos{\'e}~Antonio Carrillo, Antonio Esposito, and Jeremy Sheung-Him Wu.
\newblock Nonlocal approximation of nonlinear diffusion equations.
\newblock {\em arXiv preprint arXiv:2302.08248}, 2023.

\bibitem{Carrillo2011}
José Carrillo, Massimo Difrancesco, Alessio Figalli, T.~Laurent, and
  D.~Slep{\v c}ev.
\newblock Global-in-time weak measure solutions and finite-time aggregation for
  nonlocal interaction equations.
\newblock {\em Duke Mathematical Journal}, 156, 02 2011.

\bibitem{carrillo_chertock_huang_2015}
José~A. Carrillo, Alina Chertock, and Yanghong Huang.
\newblock A finite-volume method for nonlinear nonlocal equations with a
  gradient flow structure.
\newblock {\em Communications in Computational Physics}, 17(1):233–258, 2015.

\bibitem{Jungel_Cross2004}
Li~Chen and Ansgar J\"{u}ngel.
\newblock Analysis of a multidimensional parabolic population model with strong
  cross-diffusion.
\newblock {\em SIAM J. Math. Anal.}, 36(1):301--322, 2004.

\bibitem{ChenDausJungel_2018}
Xiuqing Chen, Esther~S Daus, and Ansgar J{\"u}ngel.
\newblock Global existence analysis of cross-diffusion population systems for
  multiple species.
\newblock {\em Archive for Rational Mechanics and Analysis}, 227(2):715--747,
  2018.

\bibitem{Colombo2012}
Rinaldo~M. Colombo and Magali L{\'e}cureux-Mercier.
\newblock Nonlocal crowd dynamics models for several populations.
\newblock {\em Acta Mathematica Scientia}, 32(1):177--196, 2012.

\bibitem{Crippa2013}
Gianluca Crippa and Magali L{\'e}cureux-Mercier.
\newblock Existence and uniqueness of measure solutions for a system of
  continuity equations with non-local flow.
\newblock {\em Nonlinear Differential Equations and Applications NoDEA},
  20(3):523--537, 2013.

\bibitem{david2023degenerate}
Noemi David, Tomasz D{\k{e}}biec, Mainak Mandal, and Markus Schmidtchen.
\newblock A degenerate cross-diffusion system as the inviscid limit of a
  nonlocal tissue growth model.
\newblock {\em arXiv preprint arXiv:2303.10620}, 2023.

\bibitem{di2018nonlinear}
Marco Di~Francesco, Antonio Esposito, and Simone Fagioli.
\newblock Nonlinear degenerate cross-diffusion systems with nonlocal
  interaction.
\newblock {\em Nonlinear Analysis}, 169:94--117, 2018.

\bibitem{Francesco2013}
Marco Di~Francesco and Simone Fagioli.
\newblock Measure solutions for non-local interaction {PDE}s with two species.
\newblock {\em Nonlinearity}, 26:2777, 09 2013.

\bibitem{Orsogna2006}
M.~R. D'Orsogna, Y.~L. Chuang, A.~L. Bertozzi, and L.~S. Chayes.
\newblock Self-propelled particles with soft-core interactions: Patterns,
  stability, and collapse.
\newblock {\em Phys. Rev. Lett.}, 96:104302, Mar 2006.

\bibitem{Elbar_JakubJDE}
Charles Elbar and Jakub Skrzeczkowski.
\newblock Degenerate {C}ahn-{H}illiard equation: {F}rom nonlocal to local.
\newblock {\em J. Differential Equations}, 364:576--611, 2023.

\bibitem{Giunta2021}
Valeria Giunta, Thomas Hillen, Mark Lewis, and Jonathan~R. Potts.
\newblock Local and global existence for nonlocal multispecies
  advection-diffusion models.
\newblock {\em SIAM Journal on Applied Dynamical Systems}, 21(3):1686--1708,
  2022.

\bibitem{JungelPortischZurek_2022}
Ansgar J{\"u}ngel, Stefan Portisch, and Antoine Zurek.
\newblock Nonlocal cross-diffusion systems for multi-species populations and
  networks.
\newblock {\em Nonlinear Analysis}, 219:112800, 2022.

\bibitem{LiebLoss}
Elliott~H. Lieb and Michael Loss.
\newblock {\em Analysis}, volume~14 of {\em Graduate Studies in Mathematics}.
\newblock American Mathematical Society, Providence, RI, second edition, 2001.

\bibitem{LMG}
Pierre-Louis Lions and Sylvie Mas-Gallic.
\newblock Une m\'{e}thode particulaire d\'{e}terministe pour des \'{e}quations
  diffusives non lin\'{e}aires.
\newblock {\em C. R. Acad. Sci. Paris S\'{e}r. I Math.}, 332(4):369--376, 2001.

\bibitem{MoussaCD20}
Ayman Moussa.
\newblock From nonlocal to classical {S}higesada-{K}awasaki-{T}eramoto systems:
  triangular case with bounded coefficients.
\newblock {\em SIAM J. Math. Anal.}, 52(1):42--64, 2020.

\bibitem{oelschlager1990large}
Karl Oelschl{\"a}ger.
\newblock Large systems of interacting particles and the porous medium
  equation.
\newblock {\em Journal of differential equations}, 88(2):294--346, 1990.

\bibitem{PHILIPOWSKI2007526}
Robert Philipowski.
\newblock Interacting diffusions approximating the porous medium equation and
  propagation of chaos.
\newblock {\em Stochastic Processes and their Applications}, 117(4):526--538,
  2007.

\bibitem{Potts2019}
Jonathan~R. Potts and Mark~A. Lewis.
\newblock Spatial memory and taxis-driven pattern formation in model
  ecosystems.
\newblock {\em Bulletin of Mathematical Biology}, 81(7):2725--2747, 2019.

\bibitem{romanczuk2012active}
Pawel Romanczuk, Markus B{\"a}r, Werner Ebeling, Benjamin Lindner, and Lutz
  Schimansky-Geier.
\newblock Active brownian particles: From individual to collective stochastic
  dynamics.
\newblock {\em The European Physical Journal Special Topics}, 202:1--162, 2012.

\end{thebibliography}

\end{document}